\documentclass[12pt]{amsart}

\usepackage{tikz}
\usepackage{amssymb}
\usepackage{amsmath}
\usepackage{comment}
\usepackage[all]{xy}

\usepackage{color}
\newcommand{\blue}[1]{\textcolor{blue}{#1}}
\newcommand{\red}[1]{\textcolor{red}{#1}}

\def\Pda{\Psi_{21}(\alpha)}

\usepackage[bookmarks=false,colorlinks,linkcolor=blue,citecolor=green]{hyperref}

\usepackage{xargs}
\usepackage[colorinlistoftodos,prependcaption,textsize=tiny,linecolor=red,backgroundcolor=red!25,bordercolor=red]{todonotes}
\newcommandx{\nantel}[2][1=]{\todo[linecolor=brown,backgroundcolor=red!25,bordercolor=blue,#1]{#2 ---Nantel}}
\newcommandx{\jc}[2][1=]{\todo[linecolor=purple,backgroundcolor=blue!25,bordercolor=red,#1]{#2 ---JC}}
\newcommandx{\john}[2][1=]{\todo[linecolor=purple,backgroundcolor=green!25,bordercolor=red,#1]{#2 ---John}}

\date{\today}

\theoremstyle{plain}
\newtheorem{theorem}{Theorem}[section]
\newtheorem{corollary}[theorem]{Corollary}
\newtheorem{lemma}[theorem]{Lemma}
\newtheorem{proposition}[theorem]{Proposition}
\newtheorem{conjecture}[theorem]{Conjecture}
\newtheorem{question}[theorem]{Question}

\theoremstyle{definition}
\newtheorem{example}[theorem]{Example} 
\newtheorem{remark}[theorem]{Remark} 
\newtheorem{definition}[theorem]{Definition} 

\theoremstyle{remark}

\specialcomment{moredetails}{\begingroup\color{red}}{\endgroup}

\newcommand{\xyinc}{\ar@{^{(}->}}

\def\field{\Bbbk}

\newcommand{\Ptwo}{\!\!\!\begin{tikzpicture}[scale=.3,baseline=-.05cm] \node (c) at (0,0) {$\scriptscriptstyle 1$}; \node (d) at (.5,0) {$\scriptscriptstyle 2$};  \end{tikzpicture}\!\!\! }
\newcommand{\Gtwo}{\!\!\! \begin{tikzpicture}[scale=.5,baseline=-.05cm] \node (c) at (0,0) {$\scriptscriptstyle 1$}; \node (d) at (.6,0) {$\scriptscriptstyle 2$}; \draw (.15,0) --  (.4,0);  \end{tikzpicture}\!\!\!}
\newcommand{\GGtwo}{\scriptstyle \bullet-\bullet}

\newcommand{\PF}{\mathcal{PF}}
\newcommand{\area}{\mathop{\mathrm{area}}}
\newcommand{\dinv}{\mathop{\mathrm{dinv}}}
\newcommand{\ides}{\mathop{\mathrm{ides}}}

\newcommand{\comp}{\mathop{\mathrm{comp}}}
\newcommand{\type}{\mathop{\mathrm{type}}}

 \author{Jean-Christophe Aval}\address[Aval]
{LaBRI, CNRS, Université de Bordeaux, FRANCE}
\email{aval@labri.fr}

\author{Nantel Bergeron}\address[Bergeron]
{Department of Mathematics and Statistics\\ York  University\\ To\-ron\-to, Ontario M3J 1P3\\ CANADA}
\email{bergeron@mathstat.yorku.ca}
\urladdr{http://www.math.yorku.ca/bergeron}
\thanks{With partial support of Bergeron's York University Research Chair and NSERC}

\author{John Machacek}\address[Machacek]
{Department of Mathematics and Statistics\\ York  University\\ To\-ron\-to, Ontario M3J 1P3\\ CANADA}
\email{machacek@yorku.ca}
\urladdr{}
\thanks{With support of York Science Fellow, York University}

\title[New Invariants for Permutations, Orders and Graphs]{New Invariants for Permutations, Orders and Graphs}

\keywords{Combinatorial Hopf algebra, combinatorial invariants, coloring, positivity} 

\subjclass[2010]{16T30; 05E05; 05E15; 05C15}

\begin{document}

\begin{abstract} We study the symmetric function and polynomial combinatorial invariants of Hopf algebras of permutations, posets and graphs.
We investigate their properties and the relations among them.
In particular, we show that the chromatic symmetric function and many other invariants have a property we call positively $h$-alternating.
This property of positively $h$-alternating leads to Schur positivity and $e$-positivity when applying the operator $\nabla$ at $q=1$.
We conclude by showing that the invariants we consider can be expressed as scheduling problems.
\end{abstract}

\maketitle

\section*{Introduction}

One of the key results in the theory of \emph{Combinatorial Hopf algebras (CHAs)} gives us a canonical way of constructing combinatorial invariants with values in the space $QSym$ of quasisymmetric functions (see~\cite{ABS}).
Let $\field$ be a field and let $H=\bigoplus_{n\ge 0} H_n$ be a CHA over $\field$ with \emph{character} (i.e. algebra morphism) $\zeta\colon H\to\field$.
We then have a unique Hopf morphism
  $$\Psi:H\to QSym$$
  such that $\zeta=\phi_1\circ\Psi$ where $\phi_1\big(f(x_1,x_2,\ldots)\big)=f(1,0,0,\ldots)$. Moreover, there is a Hopf
  morphism $\phi_t \colon QSym\to \field[t]$ given by 
   $\phi_t (M_{\alpha})={t \choose \ell(a)}$.
   Here $M_{\alpha}$ is the monomial quasisymmetric function indexed by an integer composition $\alpha=(\alpha_1,\alpha_2,\ldots,\alpha_\ell)$ and $\ell(\alpha)=\ell$ is the length of the composition.
 This Hopf morphism has the property that 
   $$\phi_t\big(f(x_1,x_2,\ldots)\big) \Big|_{t=1}=\phi_1(f)\,.$$
In particular,
   $$\phi_t\circ\Psi\Big|_{t=1}=\left( \phi_t\big|_{t=1}\right)\circ\Psi=\phi_1\circ\Psi=\zeta.$$
   
In the case when $H=\overline{\mathcal G}$, a Hopf algebra on isomorphism classes of graphs, one can define the character
  $$\zeta(G)=\begin{cases}
     1&\text{if $G$ has no edges,}\\
     0&\text{otherwise,}
  \end{cases}$$
which gives us, as shown in \cite[Example 4.5]{ABS}, that $\Psi(G)$ is the chromatic symmetric function~\cite{Sta95} and $\chi_G(t)=\phi_t\circ \Psi(G)$ is the chromatic polynomial of $G$.
Stanley's theorem on acyclic orientations~\cite{Sta1973} can be obtained pointing out that
the antipode of $\field[t]$ is given by $S\big(p(t)\big)=p(-t)$. 
Hence,
  $$\chi_G(-1)=S\circ\phi_t\circ \Psi(G)\Big|_{t=1} = \phi_t\circ \Psi\circ S(G)\Big|_{t=1}=\zeta\circ S(G).$$
Using the fact that the graph with no edges has coefficient $(-1)^{n}a(G)$ in $S(G)$ (see~\cite{Aguiar-Ardila,Benedetti-Sagan,Humpert-Martin}), where $n$ is the number of vertices of $G$ and $a(G)$ is the number of acyclic orientations of $G$, we see that $\chi_G(-1)=\zeta\circ S(G)=(-1)^{n}a(G)$.
This example with $H =\overline{ \mathcal G}$ is one instance of a general pattern where CHAs give a natural way of studying (quasi)symmetric function and polynomial invariants.
In this setting, a formula for the antipode can be thought of as a \emph{combinatorial reciprocity} result.

CHAs allow for the investigation of more general symmetric function and polynomial invariants.
One direction of generalization is to replace the Hopf algebra of graphs with a Hopf algebra of a more general class of objects containing graphs.
This gives an approach to studying a chromatic symmetric function of more general objects via CHAs.
This direction has been pursued for building sets~\cite{Gru2012}, simplicial complexes~\cite{BenedettiHallamMachacek}, and hypergraphs~\cite{Gru2016}.
For hypergraphs the antipode formula is more complicated than for graphs~\cite{BB2017, BBM2019} and so we do not obtain as nice of a reciprocity theorem.

Our focus will be on a different direction of generalization where invariants are obtained by changing the character $\zeta\colon H\to  \field$.
In addition to considering graphs, we will look at Hopf algebra structures on permutations and posets as well.
For the Hopf algebra structure along with choice of characters for permutations, posets, and graphs we have CHA morphisms
\[\{\text{Permutations}\} \to \{\text{Posets}\} \to \{\text{Graphs}\}\]
allowing us to studying symmetric function and polynomial invariants at different levels.

In Section~\ref{s:hopf} we will provide a brief review of CHAs and give the definitions of the CHAs on permutations, posets, and graphs which will be used in this paper.
This will leave us with many potential CHAs and corresponding invariants to study.
In Section~\ref{s:inv21} we will focus on a particular choice of character which leads to a symmetric function analogous to Stanley's chromatic symmetric function except we partition the vertex set into matchings rather than independent sets.
Proposition~\ref{prop:bond-21}  gives an expansion of this symmetric function in the power sum basis.
We also look at expansions in the basis of elementary symmetric functions.
In Section~\ref{s:nabla} we will focus on the usual chromatic symmetric and effect of applying $\nabla$.
Here $\nabla$ is the operator arising the study of Macdonald Polynomials and diagonal harmonics~\cite{nablaDef}.
We find in Corollary~\ref{cor:csf} that the specialization at $q=1$ of $\nabla$ applied to the chromatic symmetric function is Schur positive and in fact $e$-positive.
This $e$-positivity applies to all symmetric functions which ``positively alternate'' in the basis of homogenous symmetric functions, which is a property implied to positivity in the power sum basis after applying the involution $\omega$.
In Section~\ref{s:gen} we show that certain properties, including the previously mentioned Schur and $e$-positivity, are shared by many of the symmetric function invariants coming from the CHAs defined in this paper. We conclude by showing that the invariant we consider can be expressed as scheduling problems as in~\cite{sched, MachacekEJC}.

\section{Invariants in the Combinatorial Hopf Algebras $\mathcal S$, $\mathcal P$ and $\mathcal G$ }\label{s:hopf}

We review basic notions on Combinatorial Hopf algebras and illustrate definitions with  three classical examples. We encourage the reader to see \cite{ABS,AM:2010,BB2017} for a deeper study on this topic.

\subsection{Combinatorial Hopf Algebras (CHA)}\label{ss:CHA}
A graded connected Hopf algebra is a tuple $H=(H,m,u,\Delta,\epsilon,S)$ with the following properties 
\begin{enumerate}
\item $H=\bigoplus_{n\ge 0}H_n$ is a graded vector spaces where $\dim(H_n)<\infty$ and $\dim(H_0)=1$.
\item $m=\sum_{a,b} m_{a,b}$ and $u=\sum_n u_n$ are maps defined from
  $$m_{a,b}\colon H_a\otimes H_b\to H_{a+b}\quad \hbox{and} \quad u_n\colon \field\to H_n$$
  where $u_n=0$ for $n>0$, are such that $(H,m,u)$ is a graded algebra with multiplication $m$ and unit $u(1)$.
\item $\Delta=\sum_{a,b} \Delta_{a,b}$ and $\epsilon=\sum_n \epsilon_n$ are map defined from 
  $$\Delta_{a,b}\colon H_{a+b}\to H_a\otimes H_b \quad \hbox{and} \quad \epsilon_n\colon H_n\to\field $$
  where $\epsilon_n=0$ for $n>0$, are such that $(H,\Delta,\epsilon)$ is a graded coalgebra with comultiplication $\Delta$ and counit $\epsilon$.
\item The maps $\Delta$ and $\epsilon$ are algebra morphisms.
\item The map $S=\sum_n S_n$ defined from $S\colon H_n\to H_n$ is the antipode. That is the convolution inverse of the identity map $\hbox{Id}\colon H \to H$.
\end{enumerate}

For graded connected bialgebras (satisfying (1),(2),(3),(4) above) the existence of the antipode $S$ is guaranteed by Takeuchi's formula~\cite{Tak}:
For $x\in H_n$
\begin{equation}\label{eq:takeuchi}
 S_n(x)= \sum_{\alpha\models n} (-1)^{\ell(\alpha)} \mu_\alpha \Delta_{\alpha}(x)
\end{equation}
where the sum is over $\alpha=(a_1,a_2,\ldots,a_\ell)\models n$ compositions of $n$, the length if $\ell(\alpha)=\ell$ and
  $$m_\alpha\colon H_{a_1}\otimes H_{a_2}\otimes\cdots\otimes H_{a_\ell}\to H_n,$$
  $$\Delta_\alpha \colon H_n \to H_{a_1}\otimes H_{a_2}\otimes\cdots\otimes H_{a_\ell},$$
are maps defined recursively: 
 $$m_\alpha=\begin{cases} \mbox{Id} &\mbox{if } \ell(\alpha)=1, \\ m_{a_1,n-a_1}\circ \big(m_{(a_1)} \otimes m_{(a_2,\ldots,a_\ell)}\big)&\mbox{otherwise,} \end{cases} $$
 $$\Delta_\alpha=\begin{cases} \mbox{Id} &\mbox{if } \ell(\alpha)=1, \\ \big(\Delta_{(a_1)} \otimes \Delta_{(a_2,\ldots,a_\ell)}\big)\circ\Delta_{a_1,n-a_1} &\mbox{otherwise.} \end{cases} $$
From associativity, this definition is independent of how we actually define $m_\alpha$ and $\Delta_\alpha$.

A \emph{Combinatorial Hopf Algebra}, in the sense of \cite{ABS}, is a pair $(H,\zeta)$ where $H$ is a graded connected Hopf algebra and $\zeta\colon\field\to H$ is a \emph{character} (a multiplicative linear map). 
The most central combinatorial Hopf algebra in this sense is the pair $(QSym,\varphi_1)$ where $QSym=\bigoplus_{n\ge 0} QSym_n$ and $QSym_n$ is the $\field$-span of $\{ M_\alpha : \alpha\models n\}$.
Elements of $QSym$ are known as quasisymmetric functions.
Here $M_\alpha$ can be viewed as the following series
  $$M_\alpha (x_1,x_2,\ldots) = \sum_{i_1<i_2<\ldots<i_{\ell(\alpha)}} x_{i_1}^{a_1}x_{i_2}^{a_2}\ldots x_{i_\ell}^{a_\ell}$$
  in $\field[\![x_1,x_2,\ldots]\!]$.
The multiplication of quasisymmetric functions is the usual multiplication of series. We now define
  $$\Delta(M_\alpha)=\sum_{i=0}^{\ell(\alpha)}  M_{(a_1,\ldots,a_i)}\otimes M_{(a_{i+1},\ldots,a_\ell)}$$
and the counit $\epsilon(M_\alpha)=M_\alpha(0,0,\ldots)$. This gives us a structure of bialgebra on $QSym$ and hence makes $QSym$ a connected graded Hopf algebra.
The map $\varphi_1$ is given by the evaluation $\varphi_1(M_\alpha)=M_\alpha(1,0,0,\ldots)$.
The CHA $QSym$ contains the Hopf subalgebra $Sym$ of symmetric functions. The monomial symmetric function $m_\lambda$ is given by
  $$m_\lambda=\sum_{\alpha\in \lambda} M_\alpha$$
  where $\lambda$ is an equivalence class of compositions under the permutation of its entries. This constructs the usual monomial basis of $Sym$
  and gives us an inclusion $Sym\subset QSym$.
The following fundamental theorem of Combinatorial Hopf Algebra states that $(QSym,\varphi_1)$ is a terminal object.

\begin{theorem}[\cite{ABS}]\label{thm:ABS}
For any CHA $(H,\zeta)$ there is a unique graded morphism of Hopf algebra $\Psi_{\zeta}\colon H\to QSym$ such that $\zeta=\varphi_1\circ\Psi$.
Moreover, if $H$ is cocommutative then $\Psi\colon H\to Sym$.
\end{theorem}

It will be useful for us to recall the form of the morphism in the theorem.
Given a CHA $(H,\zeta)$ and $h \in H_n$ the morphism $\Psi_{\zeta}$ is given by
\[\Psi_{\zeta}(h) = \sum_{\alpha \vDash n} \zeta_{\alpha}(h) M_{\alpha}\]
where $\zeta_{\alpha}$ is the composition of functions
\[H \xrightarrow{\Delta^{\ell - 1}} H^{\otimes \ell} \xrightarrow{\text{projection}} H_{\alpha_1} \otimes \cdots \otimes H_{\alpha_{\ell}} \xrightarrow{\zeta^{\otimes \ell}} \field.\]
Intuitively one should think of the comultiplication $\Delta$ splitting the element $h$ while $\zeta$ selects a ``good'' splitting of $h$ and $M_\alpha$ records the grading.
\subsection{The Combinatorial Hopf Algebra of Permutation $\mathcal S$}\label{ss:perm}

Let ${\mathcal S}_n$ denote the $\field$-span of all permutation of $n$ written in one line notation.
An element $\sigma\in {\mathcal S}_n$ is given by a list $\sigma=\big(\sigma(1),\sigma(2),\ldots,\sigma(n)\big)$.
We sometimes omit all parenthesis and comas in the list describing a permutation.
For instance, if $n=3$,
\[
{\mathcal S}_3=\field \{123,\,132,\,213,\,231,\,312,\,321\}.
\]
The graded vector space ${\mathcal S}=\bigoplus_{n\ge 0} {\mathcal S}_n$ as a structure of CHA with the following operations.

By convention ${\mathcal S}_0=\field \{\emptyset\}$ where $\emptyset$ is the unique permutation of the empty set. 
Our unit is the map $u\colon\field \to {\mathcal S}$ defined by $u_0(1)=\emptyset$. The graded multiplication is given by
 $$ 
 \begin{array}{rcl}
 m_{a,b}\colon {\mathcal S}_a\otimes {\mathcal S}_b&\to&{\mathcal S}_{a+b}\\
 \alpha\otimes\beta &\mapsto & \alpha\cdot\beta^{\uparrow^a}
 \end{array}
 $$
where for $\alpha$ and $\beta$ two permutations, $\alpha\cdot\beta^{\uparrow^a}=(\alpha(1),\ldots,\alpha(a),\beta(1)+a,\ldots,\beta(b)+a)$
and the map $m_{a,b}$ is the linear extension of this operation. 

A \emph{set composition} of a finite set $I$ is a finite sequence $(A_1,\ldots,A_k)$
of disjoint non-empty subsets of $I$ whose union is $I$. In this situation, we write
$
(A_1,\ldots,A_k)\models I.
$
Let $(A_1,A_2)\models \{1,2,\ldots,n\}$. Given a permutation $\alpha\in{\mathcal S}_n$ we let $\alpha|_{A_i}$ to be the sequence 
  $$\alpha|_{A_i}=\big(\alpha(j): \alpha(j)\in A_i \hbox{ and } 1\le j\le n\big).$$
  That is we restrict the list $\alpha$ to the entries in $A_i$ only. Now there is a unique order preserving map $st\colon A_i\to \{1,2,\ldots ,|A_i|\}$,
  it is called the \emph{standardization map}. This map is	 unique such that for all $a,b\in A_i$, we have $a<b$ implies $st(a)<st(b)$.
  We can now standardize $st\big(\alpha|_{A_i}\big)=\big(st(\alpha(j)): \alpha(j)\in A_i \hbox{ and } 1\le j\le n\big)$. Technically this map depends on $A_i$, but for simplicity we abuse notation and denote by $st$ all such maps.
  The graded comultiplication of $\mathcal  S$ is now given by the map
 $$ 
 \begin{array}{rcl}
 \Delta_{a,b}\colon {\mathcal S}_{a+b}&\to&{\mathcal S}_{a}\otimes {\mathcal S}_b\\
 \alpha &\mapsto & \displaystyle \sum_{(A_1,A_2)\models \{1,2,\ldots,n\} \atop  |A_1|=a,\ |A_2|=b}  st\big(\alpha|_{A_1}\big)\otimes st\big(\alpha|_{A_2}\big).
 \end{array}
 $$
 If $a=0$, then $\Delta_{0,b} = \emptyset \otimes Id$ and similarly if $b=0$.
 The graded counit is defined by $\epsilon_0(\emptyset)=1$. With these operations, ${\mathcal S}$ is a graded connected bialgebra and hence a graded connected Hopf algebra.
 We remark that $\mathcal S$ is not commutative but is cocommutative.
 
 \begin{example} \label{ex:permops}
 Given $\alpha=31425$ and $\beta=2413$ we have 
 $$m_{5,4}(\alpha\otimes\beta)=\alpha\cdot\beta^{\uparrow^5}=314257968,$$
 $$\Delta_{1,3}(\beta)=1\otimes 312+ 1\otimes 213 + 1\otimes 132 + 1\otimes 231,$$
  $$\Delta_{2,2}(\beta)=2(12\otimes 12)+ 2(21\otimes 21) + 2(12\otimes 21)$$
 \end{example}
 
 Given a permutation $\alpha\in {\mathcal S}_n$, we say that $1\le i\le n-1$ is a \emph{global ascent} of $\alpha$ if for all $1\le r\le i<s\le n$ we have $\alpha(r)<\alpha(s)$.
 It is clear from the definition of multiplication that $i$ is a global ascent of alpha if and only if 
   $$\alpha = \alpha|_{\{1,2,\ldots,i\}}\cdot \alpha|_{\{i+1,\ldots,n\}}$$
  We say that a permutation is \emph{indecomposable} if $\alpha$ has no global ascents. It is not difficult to see from the characterization above that $\mathcal S$ is a free algebra
  and the generators are precisely the $\alpha$ with no global ascents. A complete list of free generators starts with
    $$1, 21, 231, 312, 321, \ldots $$
A character $\zeta\colon{\mathcal S}\to\field$ is completely determined by its values on the free generators. Here we will be interested in the simplest (non-zero) characters.
Namely characters $\zeta_1$, $\zeta_{21}$ or more generally $\zeta_{\gamma}$ for $\gamma$ with no global ascents which we  now define. The character $\zeta_1$ is defined by declaring $\zeta_1(1)=1$ and then setting $\zeta_1$ to zero on all other generators.
This defines $\zeta_1$ for all permutations as follows:
  $$\zeta_1(\alpha) =
    \begin{cases}
      1&\text{if } \alpha=12\cdots n  \text{ or } \alpha=\emptyset\,,\\
      0&\text{otherwise.}
    \end{cases}$$
Similarly, we let $\zeta_{21}(21)=1$ and set $\zeta_{21}$ to zero for all other generators.
This defines $\zeta_{21}$ for all permutations as follows:
  $$\zeta_{21}(\alpha) =
    \begin{cases}
      1&\text{if } \alpha=2143\cdots (2n)(2n-1)  \text{ or } \alpha=\emptyset\,,\\
      0&\text{otherwise.}
    \end{cases}$$
We find that $\zeta_{21}(\alpha)=0$ unless $\alpha\in{\mathcal S}_{2n}$ is a permutation on an even number of entries.
It is now clear that for any $\gamma$ which is a permutation with no global ascents, we can define a character with $\zeta_\gamma(\gamma)=1$ and set $\zeta_{\gamma}$ to zero for all other generators.

For each generator $\gamma$ we have a CHA given by $({\mathcal S},\zeta_\gamma)$. From Theorem~\ref{thm:ABS}, this will construct some combinatorial invariants
given by a Hopf morphism
$\Psi_\gamma\colon {\mathcal S}\to Sym$. As we will see below, for a permutation $\alpha$, the symmetric function $\Psi_1(\alpha)$ is well known as it is Stanley chromatic symmetric function for a certain graph (constructed from $\alpha$). The study of the more general invariants $\Psi_\gamma$ was proposed in~\cite{BB2017}.

\subsection{The Combinatorial Hopf Algebra of posets $\mathcal P$}\label{ss:poset} 

Let ${\mathcal P}_n$ denote the $\field$-span of all partial orders (posets) on the set $[n]=\{1,2,\ldots, n\}$. 
An element $P\in   {\mathcal P}_n$ is given by a transitive, antisymmetric, reflexive relation $<_P$ on $[n]$.
We represent $P$ by its Hasse diagram showing the cover relations.
For instance, if $n=3$,
$$ \begin{array}{c}
{\mathcal P}_3 = \field\Big\{ 
\begin{tikzpicture}[scale=.5,baseline=.3cm]
	\node (a) at (0,0) {$\scriptscriptstyle 1$}; \node (b) at (0,1) {$\scriptscriptstyle 2$}; \node (c) at (0,2) {$\scriptscriptstyle 3$};
	\draw (a) --  (b);  	\draw (b) --  (c);  \end{tikzpicture} ,
\begin{tikzpicture}[scale=.5,baseline=.3cm]
	\node (a) at (0,0) {$\scriptscriptstyle 1$}; \node (b) at (0,1) {$\scriptscriptstyle 3$}; \node (c) at (0,2) {$\scriptscriptstyle 2$};
	\draw (a) --  (b);  	\draw (b) --  (c);  \end{tikzpicture} ,
\begin{tikzpicture}[scale=.5,baseline=.3cm]
	\node (a) at (0,0) {$\scriptscriptstyle 2$}; \node (b) at (0,1) {$\scriptscriptstyle 1$}; \node (c) at (0,2) {$\scriptscriptstyle 3$};
	\draw (a) --  (b);  	\draw (b) --  (c);  \end{tikzpicture} ,
\begin{tikzpicture}[scale=.5,baseline=.3cm]
	\node (a) at (0,0) {$\scriptscriptstyle 2$}; \node (b) at (0,1) {$\scriptscriptstyle 3$}; \node (c) at (0,2) {$\scriptscriptstyle 1$};
	\draw (a) --  (b);  	\draw (b) --  (c);  \end{tikzpicture} ,
\begin{tikzpicture}[scale=.5,baseline=.3cm]
	\node (a) at (0,0) {$\scriptscriptstyle 3$}; \node (b) at (0,1) {$\scriptscriptstyle 1$}; \node (c) at (0,2) {$\scriptscriptstyle 2$};
	\draw (a) --  (b);  	\draw (b) --  (c);  \end{tikzpicture} ,
\begin{tikzpicture}[scale=.5,baseline=.3cm]
	\node (a) at (0,0) {$\scriptscriptstyle 3$}; \node (b) at (0,1) {$\scriptscriptstyle 2$}; \node (c) at (0,2) {$\scriptscriptstyle 1$};
	\draw (a) --  (b);  	\draw (b) --  (c);  \end{tikzpicture} ,
\begin{tikzpicture}[scale=.5,baseline=.3cm]
	\node (a) at (0,0) {$\scriptscriptstyle 1$}; \node (b) at (-.3,1) {$\scriptscriptstyle 2$}; \node (c) at (.3,1) {$\scriptscriptstyle 3$};
	\draw (a) --  (b);  	\draw (a) --  (c);  \end{tikzpicture} ,
\begin{tikzpicture}[scale=.5,baseline=.3cm]
	\node (a) at (0,0) {$\scriptscriptstyle 2$}; \node (b) at (-.3,1) {$\scriptscriptstyle 1$}; \node (c) at (.3,1) {$\scriptscriptstyle 3$};
	\draw (a) --  (b);  	\draw (a) --  (c);  \end{tikzpicture} ,
\begin{tikzpicture}[scale=.5,baseline=.3cm]
	\node (a) at (0,0) {$\scriptscriptstyle 3$}; \node (b) at (-.3,1) {$\scriptscriptstyle 1$}; \node (c) at (.3,1) {$\scriptscriptstyle 2$};
	\draw (a) --  (b);  	\draw (a) --  (c);  \end{tikzpicture} ,
\begin{tikzpicture}[scale=.5,baseline=.3cm]
	\node (a) at (0,1) {$\scriptscriptstyle 1$}; \node (b) at (-.3,0) {$\scriptscriptstyle 2$}; \node (c) at (.3,0) {$\scriptscriptstyle 3$};
	\draw (a) --  (b);  	\draw (a) --  (c);  \end{tikzpicture} ,
\begin{tikzpicture}[scale=.5,baseline=.3cm]
	\node (a) at (0,1) {$\scriptscriptstyle 2$}; \node (b) at (-.3,0) {$\scriptscriptstyle 1$}; \node (c) at (.3,0) {$\scriptscriptstyle 3$};
	\draw (a) --  (b);  	\draw (a) --  (c);  \end{tikzpicture} ,
\begin{tikzpicture}[scale=.5,baseline=.3cm]
	\node (a) at (0,1) {$\scriptscriptstyle 3$}; \node (b) at (-.3,0) {$\scriptscriptstyle 1$}; \node (c) at (.3,0) {$\scriptscriptstyle 2$};
	\draw (a) --  (b);  	\draw (a) --  (c);  \end{tikzpicture} ,\\
\begin{tikzpicture}[scale=.5,baseline=.3cm]
	\node (a) at (0,0) {$\scriptscriptstyle 1$}; \node (b) at (.6,0) {$\scriptscriptstyle 2$}; \node (c) at (.6,1) {$\scriptscriptstyle 3$};
	\draw (b) --  (c);  \end{tikzpicture} ,
\begin{tikzpicture}[scale=.5,baseline=.3cm]
	\node (a) at (0,0) {$\scriptscriptstyle 1$}; \node (b) at (.6,0) {$\scriptscriptstyle 3$}; \node (c) at (.6,1) {$\scriptscriptstyle 2$};
	\draw (b) --  (c);  \end{tikzpicture} ,
\begin{tikzpicture}[scale=.5,baseline=.3cm]
	\node (a) at (0,0) {$\scriptscriptstyle 2$}; \node (b) at (.6,0) {$\scriptscriptstyle 1$}; \node (c) at (.6,1) {$\scriptscriptstyle 3$};
	\draw (b) --  (c);  \end{tikzpicture} ,
\begin{tikzpicture}[scale=.5,baseline=.3cm]
	\node (a) at (0,0) {$\scriptscriptstyle 2$}; \node (b) at (.6,0) {$\scriptscriptstyle 3$}; \node (c) at (.6,1) {$\scriptscriptstyle 1$};
	\draw (b) --  (c);  \end{tikzpicture} ,
\begin{tikzpicture}[scale=.5,baseline=.3cm]
	\node (a) at (0,0) {$\scriptscriptstyle 3$}; \node (b) at (.6,0) {$\scriptscriptstyle 1$}; \node (c) at (.6,1) {$\scriptscriptstyle 2$};
	\draw (b) --  (c);  \end{tikzpicture} ,
\begin{tikzpicture}[scale=.5,baseline=.3cm]
	\node (a) at (0,0) {$\scriptscriptstyle 3$}; \node (b) at (.6,0) {$\scriptscriptstyle 2$}; \node (c) at (.6,1) {$\scriptscriptstyle 1$};
	\draw (b) --  (c);  \end{tikzpicture} ,
\begin{tikzpicture}[scale=.5,baseline=.3cm]
	\node (a) at (0,0) {$\scriptscriptstyle 1$}; \node (b) at (.6,0) {$\scriptscriptstyle 2$}; \node (c) at (1.2,0) {$\scriptscriptstyle 3$}; \end{tikzpicture}
\Big\}\,.
\end{array}$$
The graded vector space ${\mathcal P}=\bigoplus_{n\ge 0} {\mathcal P}_n$ has a structure of Hopf algebra with operations we will now describe.

By convention ${\mathcal P}_0=\field \{\emptyset\}$ where $\emptyset$ is the unique poset of the empty set. 
The unit is the map $u\colon\field \to {\mathcal P}$ defined by $u_0(1)=\emptyset$. The graded multiplication is given by
 $$ 
 \begin{array}{rcl}
 m_{a,b}\colon {\mathcal P}_a\otimes {\mathcal P}_b&\to&{\mathcal P}_{a+b}\\
 P\otimes Q &\mapsto & P < Q^{\uparrow^a}
 \end{array}
 $$
where for the poset $Q$ on $\{1,2,\ldots , b\}$, the poset $Q^{\uparrow^a}$ on $\{a+1, a+2,\ldots,a+b\}$ is defined by the relation satisfying  
$(a+x)<_{Q^{\uparrow^a}}(a+y)$ if and only if $x<_{Q} y$. For two posets $P\in {\mathcal P}_a$ and $Q\in {\mathcal P}_b$, the poset $P< Q^{\uparrow^a} \in {\mathcal P}_{a+b}$ is the relation $<_P$ on 
$\{1,2,\ldots,a\}$, $<_{Q^{\uparrow^a}}$ on $\{a+1,a+2,\ldots,a+b\}$ and $i<j$ always when $i\le a<j$.
The map $m_{a,b}$ is the linear extension of this operation. 

  The graded comultiplication of $\mathcal  P$ is  given by the map
 $$ 
 \begin{array}{rcl}
 \Delta_{a,b}\colon {\mathcal P}_{a+b}&\to&{\mathcal P}_{a}\otimes {\mathcal P}_b\\
P &\mapsto & \displaystyle \sum_{(A_1,A_2)\models \{1,2,\ldots,n\} \atop  |A_1|=a,\ |A_2|=b}  st\big(P|_{A_1}\big)\otimes st\big(P|_{A_2}\big).
 \end{array}
 $$
 Where $P|_{A_i}$ is the restriction of the relation $<_P$ on the subset $A_i\subseteq \{1,2,\ldots,a+b\}$. If $a=0$ or $b=0$, then we use the same convention as for permutations.
 The graded counit is defined by $\epsilon_0(\emptyset)=1$. With these operations, ${\mathcal P}$ is a graded connected bialgebra and hence a graded  connected Hopf algebra.
Note that $\mathcal P$ is not commutative but it is cocommutative.
 
 \begin{example} \label{ex:posetops}
 Given 
 $P=\begin{tikzpicture}[scale=.5,baseline=.1cm]
	\node (a) at (0,1) {$\scriptscriptstyle 2$}; \node (b) at (0,0) {$\scriptscriptstyle 1$}; \node (c) at (.6,0) {$\scriptscriptstyle 3$};
	\draw (a) --  (b);  \end{tikzpicture}$ 
and 
$Q=\begin{tikzpicture}[scale=.5,baseline=.1cm]
	\node (a) at (0,1) {$\scriptscriptstyle 1$}; \node (b) at (0,0) {$\scriptscriptstyle 2$}; \node (c) at (.6,0) {$\scriptscriptstyle 3$}; \node (d) at (.6,1) {$\scriptscriptstyle 4$}; 
	\draw (a) --  (b);  \draw (a) --  (c);  	\draw (d) --  (c);  \end{tikzpicture}$
we have 
$$m_{3,4}(P\otimes Q)= P< Q^{\uparrow^3} =
\begin{tikzpicture}[scale=.5,baseline=.1cm]
	\node (aa) at (0,1) {$\scriptscriptstyle 2$}; \node (bb) at (0,0) {$\scriptscriptstyle 1$}; \node (cc) at (1.3,.3) {$\scriptscriptstyle 3$};
	\draw (aa) --  (bb);  
	\node (a) at (0,3) {$\scriptscriptstyle 4$}; \node (b) at (0,2) {$\scriptscriptstyle 5$}; \node (c) at (.6,2) {$\scriptscriptstyle 6$}; \node (d) at (.6,3) {$\scriptscriptstyle 7$}; 
	\draw (a) --  (b);  \draw (a) --  (c);  	\draw (d) --  (c);  
	\draw (0,1.2) --  (0,1.8);  \draw (0,1.2) --  (.6,1.8);  	\draw (1.3,.5) --  (0,1.8);  	\draw (1.3,.5) --  (.6,1.8);  
	\end{tikzpicture},$$
 $$\Delta_{1,3}(Q)=
     {\scriptscriptstyle 1}\otimes \begin{tikzpicture}[scale=.5,baseline=.1cm]
	 \node (b) at (0,0) {$\scriptscriptstyle 1$}; \node (c) at (.6,0) {$\scriptscriptstyle 2$}; \node (d) at (.6,1) {$\scriptscriptstyle 3$}; 
	\draw (d) --  (c);  \end{tikzpicture}
  + {\scriptscriptstyle 1}\otimes \begin{tikzpicture}[scale=.5,baseline=.1cm]
	\node (a) at (0,1) {$\scriptscriptstyle 1$};  \node (c) at (.6,0) {$\scriptscriptstyle 2$}; \node (d) at (.6,1) {$\scriptscriptstyle 3$}; 
	 \draw (a) --  (c);  	\draw (d) --  (c);  \end{tikzpicture}
  + {\scriptscriptstyle 1}\otimes \begin{tikzpicture}[scale=.5,baseline=.1cm]
	\node (a) at (0,1) {$\scriptscriptstyle 1$}; \node (b) at (0,0) {$\scriptscriptstyle 2$};  \node (d) at (.6,1) {$\scriptscriptstyle 3$}; 
	\draw (a) --  (b);   \end{tikzpicture}
  + {\scriptscriptstyle 1}\otimes \begin{tikzpicture}[scale=.5,baseline=.1cm]
	\node (a) at (0,1) {$\scriptscriptstyle 1$}; \node (b) at (0,0) {$\scriptscriptstyle 2$}; \node (c) at (.6,0) {$\scriptscriptstyle 3$}; 
	\draw (a) --  (b);  \draw (a) --  (c);  	 \end{tikzpicture}$$
  $$\Delta_{2,2}(Q)=
       \begin{tikzpicture}[scale=.5,baseline=.1cm] \node (a) at (0,1) {$\scriptscriptstyle 1$}; \node (b) at (0,0) {$\scriptscriptstyle 2$}; \draw (a) --  (b);  \end{tikzpicture} 
       \otimes	
       \begin{tikzpicture}[scale=.5,baseline=.1cm]  \node (c) at (0,0) {$\scriptscriptstyle 1$}; \node (d) at (0,1) {$\scriptscriptstyle 2$};  \draw (d) --  (c);  \end{tikzpicture}
 +   \begin{tikzpicture}[scale=.5,baseline=.1cm] \node (a) at (0,1) {$\scriptscriptstyle 1$}; \node (c) at (0,0) {$\scriptscriptstyle 2$}; \draw (a) --  (c);   \end{tikzpicture} 
      \otimes {\scriptscriptstyle 1\, 2} 
 +    2( {\scriptscriptstyle 1\, 2} \otimes {\scriptscriptstyle 1\, 2} )
 +     {\scriptscriptstyle 1\, 2} \otimes
	\begin{tikzpicture}[scale=.5,baseline=.1cm] \node (a) at (0,1) {$\scriptscriptstyle 1$}; \node (c) at (0,0) {$\scriptscriptstyle 2$};  \draw (a) --  (c);   \end{tikzpicture}
 +     \begin{tikzpicture}[scale=.5,baseline=.1cm] \node (c) at (0,0) {$\scriptscriptstyle 1$}; \node (d) at (0,1) {$\scriptscriptstyle 2$}; \draw (d) --  (c);  \end{tikzpicture}
 	\otimes
	\begin{tikzpicture}[scale=.5,baseline=.1cm] \node (a) at (0,1) {$\scriptscriptstyle 1$}; \node (b) at (0,0) {$\scriptscriptstyle 2$};  \draw (a) --  (b);  \end{tikzpicture}
  $$
 \end{example}

Given a poset $P\in {\mathcal P}_n$, we say that $1\le i\le n-1$ is a \emph{global split} of $P$ if  
   $$P = P|_{\{1,2,\ldots,i\}} < P|_{\{i+1,\ldots,n\}}$$
  We say that a poset is \emph{indecomposable} if $P$ has no global splits. It is not difficult to see from the characterization above that $\mathcal P$ is a free algebra
  and that the generators are precisely the posets with no global splits. A complete list of free generators start with
    $${\scriptscriptstyle 1}, 
    \begin{tikzpicture}[scale=.5,baseline=-.05cm] \node (c) at (0,0) {$\scriptscriptstyle 1$}; \node (d) at (.5,0) {$\scriptscriptstyle 2$};  \end{tikzpicture},
    \begin{tikzpicture}[scale=.5,baseline=.3cm] \node (c) at (0,0) {$\scriptscriptstyle 2$}; \node (d) at (0,1) {$\scriptscriptstyle 1$}; \draw (d) --  (c);  \end{tikzpicture},
    \begin{tikzpicture}[scale=.5,baseline=-.05cm]
	\node (a) at (0,0) {$\scriptscriptstyle 1$}; \node (b) at (.5,0) {$\scriptscriptstyle 2$}; \node (c) at (1,0) {$\scriptscriptstyle 3$};
	  \end{tikzpicture} ,
    \begin{tikzpicture}[scale=.5,baseline=.3cm]
	\node (a) at (0,0) {$\scriptscriptstyle 2$}; \node (b) at (-.3,1) {$\scriptscriptstyle 1$}; \node (c) at (.3,1) {$\scriptscriptstyle 3$};
	\draw (a) --  (b);  	\draw (a) --  (c);  \end{tikzpicture} ,
    \begin{tikzpicture}[scale=.5,baseline=.3cm]
	\node (a) at (0,0) {$\scriptscriptstyle 3$}; \node (b) at (-.3,1) {$\scriptscriptstyle 1$}; \node (c) at (.3,1) {$\scriptscriptstyle 2$};
	\draw (a) --  (b);  	\draw (a) --  (c);  \end{tikzpicture} ,
    \begin{tikzpicture}[scale=.5,baseline=.3cm]
	\node (a) at (0,1) {$\scriptscriptstyle 1$}; \node (b) at (-.3,0) {$\scriptscriptstyle 2$}; \node (c) at (.3,0) {$\scriptscriptstyle 3$};
	\draw (a) --  (b);  	\draw (a) --  (c);  \end{tikzpicture} ,
\begin{tikzpicture}[scale=.5,baseline=.3cm]
	\node (a) at (0,1) {$\scriptscriptstyle 2$}; \node (b) at (-.3,0) {$\scriptscriptstyle 1$}; \node (c) at (.3,0) {$\scriptscriptstyle 3$};
	\draw (a) --  (b);  	\draw (a) --  (c);  \end{tikzpicture} ,
\begin{tikzpicture}[scale=.5,baseline=.3cm]
	\node (a) at (0,0) {$\scriptscriptstyle 1$}; \node (b) at (.6,0) {$\scriptscriptstyle 2$}; \node (c) at (.6,1) {$\scriptscriptstyle 3$};
	\draw (b) --  (c);  \end{tikzpicture} ,
\begin{tikzpicture}[scale=.5,baseline=.3cm]
	\node (a) at (0,0) {$\scriptscriptstyle 1$}; \node (b) at (.6,0) {$\scriptscriptstyle 3$}; \node (c) at (.6,1) {$\scriptscriptstyle 2$};
	\draw (b) --  (c);  \end{tikzpicture} ,
\begin{tikzpicture}[scale=.5,baseline=.3cm]
	\node (a) at (0,0) {$\scriptscriptstyle 2$}; \node (b) at (.6,0) {$\scriptscriptstyle 1$}; \node (c) at (.6,1) {$\scriptscriptstyle 3$};
	\draw (b) --  (c);  \end{tikzpicture} ,
\begin{tikzpicture}[scale=.5,baseline=.3cm]
	\node (a) at (0,0) {$\scriptscriptstyle 2$}; \node (b) at (.6,0) {$\scriptscriptstyle 3$}; \node (c) at (.6,1) {$\scriptscriptstyle 1$};
	\draw (b) --  (c);  \end{tikzpicture} ,
\begin{tikzpicture}[scale=.5,baseline=.3cm]
	\node (a) at (0,0) {$\scriptscriptstyle 3$}; \node (b) at (.6,0) {$\scriptscriptstyle 1$}; \node (c) at (.6,1) {$\scriptscriptstyle 2$};
	\draw (b) --  (c);  \end{tikzpicture} ,
\begin{tikzpicture}[scale=.5,baseline=.3cm]
	\node (a) at (0,0) {$\scriptscriptstyle 3$}; \node (b) at (.6,0) {$\scriptscriptstyle 2$}; \node (c) at (.6,1) {$\scriptscriptstyle 1$};
	\draw (b) --  (c);  \end{tikzpicture} ,
    \begin{tikzpicture}[scale=.5,baseline=.3cm]
	\node (a) at (0,0) {$\scriptscriptstyle 2$}; \node (b) at (0,1) {$\scriptscriptstyle 3$}; \node (c) at (0,2) {$\scriptscriptstyle 1$};
	\draw (a) --  (b);  	\draw (b) --  (c);  \end{tikzpicture} ,
    \begin{tikzpicture}[scale=.5,baseline=.3cm]
	\node (a) at (0,0) {$\scriptscriptstyle 3$}; \node (b) at (0,1) {$\scriptscriptstyle 1$}; \node (c) at (0,2) {$\scriptscriptstyle 2$};
	\draw (a) --  (b);  	\draw (b) --  (c);  \end{tikzpicture} ,
    \begin{tikzpicture}[scale=.5,baseline=.3cm]
	\node (a) at (0,0) {$\scriptscriptstyle 3$}; \node (b) at (0,1) {$\scriptscriptstyle 2$}; \node (c) at (0,2) {$\scriptscriptstyle 1$};
	\draw (a) --  (b);  	\draw (b) --  (c);  \end{tikzpicture} ,
     \ldots $$
As before, a character $\zeta\colon{\mathcal P}\to\field$ is completely determined by its values on the free generators. We consider a character $\zeta_{Q}$ for each $Q$ with no global splits. We define a character by $\zeta_1(1)=1$ and setting $\zeta_1$ to zero for all other generators.
This defines $\zeta_1$ for all posets as follows:
  $$\zeta_1(P) =
    \begin{cases}
      1&\text{if } P={\scriptstyle 1<2< \cdots<n } \text{ or } P=\emptyset\,,\\
      0&\text{otherwise.}
    \end{cases}$$
Similarly, we define $\zeta_{\Ptwo}(\begin{tikzpicture}[scale=.5,baseline=-.05cm] \node (c) at (0,0) {$\scriptscriptstyle 1$}; \node (d) at (.5,0) {$\scriptscriptstyle 2$}; \end{tikzpicture})=1$ and  set $\zeta_{\Ptwo}$ to zero for all other generators.
This defines $\zeta_{\Ptwo}$ for all permutations as follows:
  $$\zeta_{\Ptwo}(P) =
    \begin{cases}
      1&\text{if } P=12<34<\cdots<(2n-1)(2n)  \text{ or } P=\emptyset\,,\\
      0&\text{otherwise.}
    \end{cases}$$
It follows that $\zeta_{\Ptwo}(P)=0$ unless $P\in{\mathcal P}_{2n}$ is a poset on an even number of entries.

For each generator $\gamma$ we have a CHA given by $({\mathcal P},\zeta_\gamma)$. From Theorem~\ref{thm:ABS}, this will construct some combinatorial invariants
given by a Hopf morphism
$\Psi_\gamma\colon {\mathcal P}\to Sym$. The fact that ${\mathcal P}$ is cocommutative gives us that the image of $\Psi_\gamma$ is inside $Sym$.

There is an embbeding of the Hopf algebra of permutation $\mathcal S$ inside $\mathcal P$. Given a permutation $\alpha \in {\mathcal S}_n$ we construct a poset $P_\alpha\in {\mathcal P}_n$ 
with the order $<_\alpha$ defined as follows:
  $$i<_\alpha j \qquad\iff\qquad \big( i<j \quad \text{and} \quad \alpha(i)<\alpha(j) \big).$$
\begin{example} For $n=3$:
$$P_{123}=\begin{tikzpicture}[scale=.5,baseline=.3cm]
	\node (a) at (0,0) {$\scriptscriptstyle 1$}; \node (b) at (0,1) {$\scriptscriptstyle 2$}; \node (c) at (0,2) {$\scriptscriptstyle 3$};
	\draw (a) --  (b);  	\draw (b) --  (c);  \end{tikzpicture} ,
\quad P_{132}=\begin{tikzpicture}[scale=.5,baseline=.15cm]
	\node (a) at (0,0) {$\scriptscriptstyle 1$}; \node (b) at (-.3,1) {$\scriptscriptstyle 2$}; \node (c) at (.3,1) {$\scriptscriptstyle 3$};
	\draw (a) --  (b);  	\draw (a) --  (c);  \end{tikzpicture} ,
\quad P_{213}=\begin{tikzpicture}[scale=.5,baseline=.15cm]
	\node (a) at (0,1) {$\scriptscriptstyle 3$}; \node (b) at (-.3,0) {$\scriptscriptstyle 1$}; \node (c) at (.3,0) {$\scriptscriptstyle 2$};
	\draw (a) --  (b);  	\draw (a) --  (c);  \end{tikzpicture} ,\\
\quad P_{231}=\begin{tikzpicture}[scale=.5,baseline=.15cm]
	\node (a) at (0,0) {$\scriptscriptstyle 1$}; \node (b) at (.6,0) {$\scriptscriptstyle 2$}; \node (c) at (.6,1) {$\scriptscriptstyle 3$};
	\draw (b) --  (c);  \end{tikzpicture} ,
\quad P_{312}=\begin{tikzpicture}[scale=.5,baseline=.15cm]
	\node (a) at (0,0) {$\scriptscriptstyle 3$}; \node (b) at (.6,0) {$\scriptscriptstyle 1$}; \node (c) at (.6,1) {$\scriptscriptstyle 2$};
	\draw (b) --  (c);  \end{tikzpicture} ,
\quad P_{321}=\begin{tikzpicture}[scale=.5,baseline=-.05cm]
	\node (a) at (0,0) {$\scriptscriptstyle 1$}; \node (b) at (.6,0) {$\scriptscriptstyle 2$}; \node (c) at (1.2,0) {$\scriptscriptstyle 3$}; \end{tikzpicture}
	$$
\end{example}
Extending linearly the map $\alpha \mapsto P_\alpha$ defines a graded Hopf embeding ${\mathcal S}\hookrightarrow {\mathcal P}$. Indeed, it is straightforward to check that
  $$P_{m_{a,b}(\alpha\otimes\beta)} =m_{a,b}(P_\alpha\otimes P_\beta)\qquad\text{and}\qquad P_{\Delta_{a,b}(\alpha)} =\Delta_{a,b}(P_\alpha).$$
Now if we have an indecomposable permutation $\gamma$, then $P_\gamma$ is an indecomposable poset. Furthermore, for $\gamma$ indecomposable, the character $\zeta_\gamma$ and 
$\zeta_{P_\gamma}$ are related as follows:
  $$ \zeta_\alpha(\beta)=\zeta_{P_\alpha}(P_\beta).$$
Hence, the unicity of the morphism in Theorem~\ref{thm:ABS}  gives us that $\Psi_\gamma(\beta)=\Psi_{P_\gamma}(P_\beta)$. We summarize this in the following Lemma.
\begin{lemma} Let $\varphi\colon {\mathcal S}\hookrightarrow {\mathcal P}$ be the Hopf morphism defined by $\varphi(\beta)=P_\beta$. For $\gamma$ an indecomposable permutation of $\mathcal S$,
we have that
\begin{enumerate}
\item[1.] $\varphi(\gamma)$ is an indecomposable poset of $\mathcal P$
\item[2.] The morphisms $\Psi_\gamma\colon{\mathcal S}\to Sym$ and $\Psi_{\varphi(\gamma)}\colon{\mathcal P}\to Sym$ satisfies $\Psi_\gamma=\Psi_{\varphi(\gamma)}\circ\varphi$.
\end{enumerate}
\end{lemma}
This shows that the invariant $\Psi_\gamma$ is a particular instance of a poset invariant. In the next subsection we will see a natural morphism of the Hopf algebra $\mathcal P$ inside the Hopf algebra of graphs. This will show that the invariants above are in fact graph invariants. In particular $\Psi_1$ is the chromatic symmetric function of Stanley.

\subsection{The Combinatorial Hopf Algebra of graphs $\mathcal G$}\label{ss:graph}

We now turn to labelled graphs.
Let ${\mathcal G}_n$ denote the $\field$-span of all simple graphs on the set $[n]=\{1,2,\ldots, n\}$. 
An element ${\bf g}\in   {\mathcal G}_n$ is determined by a subset of ${[n]\choose 2}$, the set of 2-subsets of $[n]$.
That is the set of {\it edges} of $\bf g$.
We represent ${\bf g}$ by drawing a line joining $i$ and $j$ if and only if $\{i,j\}\in{\bf g}$.
For instance, if $n=3$,
$$ 
{\mathcal G}_3 = \field\Big\{ 
\begin{tikzpicture}[scale=.5,baseline=.1cm]
	\node (a) at (0,0) {$\scriptscriptstyle 1$}; \node (b) at (-.5,1) {$\scriptscriptstyle 2$}; \node (c) at (.5,1) {$\scriptscriptstyle 3$};
	\draw (a) --  (b);  	\draw (b) --  (c);  \draw (a) --  (c);  \end{tikzpicture} ,
\begin{tikzpicture}[scale=.5,baseline=.1cm]
	\node (a) at (0,0) {$\scriptscriptstyle 1$}; \node (b) at (-.5,1) {$\scriptscriptstyle 2$}; \node (c) at (.5,1) {$\scriptscriptstyle 3$};
		\draw (b) --  (c);  \draw (a) --  (c);  \end{tikzpicture} ,
\begin{tikzpicture}[scale=.5,baseline=.1cm]
	\node (a) at (0,0) {$\scriptscriptstyle 1$}; \node (b) at (-.5,1) {$\scriptscriptstyle 2$}; \node (c) at (.5,1) {$\scriptscriptstyle 3$};
	\draw (a) --  (b);  	 \draw (a) --  (c);  \end{tikzpicture} ,
\begin{tikzpicture}[scale=.5,baseline=.1cm]
	\node (a) at (0,0) {$\scriptscriptstyle 1$}; \node (b) at (-.5,1) {$\scriptscriptstyle 2$}; \node (c) at (.5,1) {$\scriptscriptstyle 3$};
	\draw (a) --  (b);  	\draw (b) --  (c);  \end{tikzpicture} ,
\begin{tikzpicture}[scale=.5,baseline=.1cm]
	\node (a) at (0,0) {$\scriptscriptstyle 1$}; \node (b) at (-.5,1) {$\scriptscriptstyle 2$}; \node (c) at (.5,1) {$\scriptscriptstyle 3$};
	 \draw (a) --  (c);  \end{tikzpicture} ,
\begin{tikzpicture}[scale=.5,baseline=.1cm]
	\node (a) at (0,0) {$\scriptscriptstyle 1$}; \node (b) at (-.5,1) {$\scriptscriptstyle 2$}; \node (c) at (.5,1) {$\scriptscriptstyle 3$};
	\draw (b) --  (c);  \end{tikzpicture} ,
\begin{tikzpicture}[scale=.5,baseline=.1cm]
	\node (a) at (0,0) {$\scriptscriptstyle 1$}; \node (b) at (-.5,1) {$\scriptscriptstyle 2$}; \node (c) at (.5,1) {$\scriptscriptstyle 3$};
	\draw (a) --  (b);  	 \end{tikzpicture} ,
\begin{tikzpicture}[scale=.5,baseline=.1cm]
	\node (a) at (0,0) {$\scriptscriptstyle 1$}; \node (b) at (-.5,1) {$\scriptscriptstyle 2$}; \node (c) at (.5,1) {$\scriptscriptstyle 3$};
	 \end{tikzpicture} 
\Big\}\,.
$$
The graded vector space ${\mathcal G}=\bigoplus_{n\ge 0} {\mathcal G}_n$ as a structure of CHA with the following operations.

Similar to before ${\mathcal G}_0=\field \{ \emptyset\}$ where $\emptyset$ is the unique graph of the empty set. 
The unit is the map $u\colon\field \to {\mathcal G}$ defined by $u_0(1)=\emptyset$. The graded multiplication is given by
 $$ 
 \begin{array}{rcl}
 m_{a,b}\colon {\mathcal G}_a\otimes {\mathcal G}_b&\to&{\mathcal  G}_{a+b}\\
 {\bf g}\otimes  {\bf h} &\mapsto &  {\bf g}\cup  {\bf h}^{\uparrow^a}
 \end{array}
 $$
where for the graph $ {\bf h}$ on $\{1,2,\ldots , b\}$, the subset $ {\bf h}^{\uparrow^a}\subseteq{\{a+1, a+2,\ldots,a+b\}\choose 2}$ is defined by   
$\{i,j\}\in  {\bf h}$ if and only if $\{i+a,j+a\}\in  {\bf h}^{\uparrow^a}$. 
The union ${\bf g}\cup  {\bf h}^{\uparrow^a}$ is clearly a subset of ${[a+b]\choose 2}$.
The map $m_{a,b}$ is the linear extension of this operation. 

  The graded comultiplication of $\mathcal  G$ is  given by the map
 $$ 
 \begin{array}{rcl}
 \Delta_{a,b}\colon {\mathcal G}_{a+b}&\to&{\mathcal G}_{a}\otimes {\mathcal G}_b\\
{\bf g} &\mapsto & \displaystyle \sum_{(A_1,A_2)\models \{1,2,\ldots,n\} \atop  |A_1|=a,\ |A_2|=b}  st\big({\bf g}|_{A_1}\big)\otimes st\big({\bf g}|_{A_2}\big),
 \end{array}
 $$
 where ${\bf g}|_{A_k}=\big\{ \{i,j\}\in {\bf g} : \{i,j\}\subseteq A_k\big\}$.  For $a=0$ or $b=0$ we use the same convention as before.
 The graded counit is defined by $\epsilon_0(\emptyset)=1$. With these operations, ${\mathcal G}$ is a graded connected bialgebra and hence a graded connected Hopf algebra.
 Again we have a Hopf algebra which is cocommutative but not commutative.
 
 \begin{example} \label{ex:graphops}
 Given 
 ${\bf g}=\begin{tikzpicture}[scale=.5,baseline=.1cm]
	\node (a) at (0,1) {$\scriptscriptstyle 2$}; \node (b) at (0,0) {$\scriptscriptstyle 1$}; \node (c) at (.6,.5) {$\scriptscriptstyle 3$};
	\draw (a) --  (b); \draw (a)--(c); \end{tikzpicture}$ 
and 
${\bf h}=\begin{tikzpicture}[scale=.5,baseline=.1cm]
	\node (a) at (0,1) {$\scriptscriptstyle 1$}; \node (b) at (0,0) {$\scriptscriptstyle 2$}; \node (c) at (1,0) {$\scriptscriptstyle 3$}; \node (d) at (1,1) {$\scriptscriptstyle 4$}; 
	\draw (a) --  (b);  \draw (a) --  (c);  	\draw (d) --  (c); \draw (a) -- (d); \end{tikzpicture}$
we have 
$$m_{3,4}({\bf g}\otimes {\bf h})= {\bf g}\cup{\bf h}^{\uparrow^3} =
        \begin{tikzpicture}[scale=.5,baseline=.1cm]
	\node (a) at (0,1) {$\scriptscriptstyle 2$}; \node (b) at (0,0) {$\scriptscriptstyle 1$}; \node (c) at (.6,.5) {$\scriptscriptstyle 3$};
	\draw (a) --  (b); \draw (a)--(c); \end{tikzpicture}
	\begin{tikzpicture}[scale=.5,baseline=.1cm]
	\node (a) at (0,1) {$\scriptscriptstyle 4$}; \node (b) at (0,0) {$\scriptscriptstyle 5$}; \node (c) at (1,0) {$\scriptscriptstyle 6$}; \node (d) at (1,1) {$\scriptscriptstyle 7$}; 
	\draw (a) --  (b);  \draw (a) --  (c);  	\draw (d) --  (c); \draw (a) -- (d); \end{tikzpicture}
	$$
 $$\Delta_{1,3}({\bf h})=
     {\scriptscriptstyle 1}\otimes \begin{tikzpicture}[scale=.5,baseline=.1cm]
	 \node (b) at (0,0) {$\scriptscriptstyle 1$}; \node (c) at (1,0) {$\scriptscriptstyle 2$}; \node (d) at (1,1) {$\scriptscriptstyle 3$}; 
	\draw (d) --  (c);  \end{tikzpicture}
  + {\scriptscriptstyle 1}\otimes \begin{tikzpicture}[scale=.5,baseline=.1cm]
	\node (a) at (0,1) {$\scriptscriptstyle 1$};  \node (c) at (1,0) {$\scriptscriptstyle 2$}; \node (d) at (1,1) {$\scriptscriptstyle 3$}; 
	 \draw (a) --  (c);  	\draw (d) --  (c); \draw (a) -- (d);  \end{tikzpicture}
  + {\scriptscriptstyle 1}\otimes \begin{tikzpicture}[scale=.5,baseline=.1cm]
	\node (a) at (0,1) {$\scriptscriptstyle 1$}; \node (b) at (0,0) {$\scriptscriptstyle 2$};  \node (d) at (1,1) {$\scriptscriptstyle 3$}; 
	\draw (a) --  (b);  \draw (a) -- (d);  \end{tikzpicture}
  + {\scriptscriptstyle 1}\otimes \begin{tikzpicture}[scale=.5,baseline=.1cm]
	\node (a) at (0,1) {$\scriptscriptstyle 1$}; \node (b) at (0,0) {$\scriptscriptstyle 2$}; \node (c) at (1,0) {$\scriptscriptstyle 3$}; 
	\draw (a) --  (b);  \draw (a) --  (c);  	 \end{tikzpicture}$$
  $$\Delta_{2,2}({\bf h})=
       2\big(\begin{tikzpicture}[scale=.5,baseline=.1cm] \node (a) at (0,1) {$\scriptscriptstyle 1$}; \node (b) at (0,0) {$\scriptscriptstyle 2$}; \draw (a) --  (b);  \end{tikzpicture} 
       \otimes	
       \begin{tikzpicture}[scale=.5,baseline=.1cm]  \node (c) at (0,0) {$\scriptscriptstyle 2$}; \node (d) at (0,1) {$\scriptscriptstyle 1$};  \draw (d) --  (c);  \end{tikzpicture}\big)
 +   \begin{tikzpicture}[scale=.5,baseline=.1cm] \node (a) at (0,1) {$\scriptscriptstyle 1$}; \node (c) at (.6,0) {$\scriptscriptstyle 2$}; \draw (a) --  (c);   \end{tikzpicture} 
      \otimes  \begin{tikzpicture}[scale=.5,baseline=.1cm] \node (a) at (0,0) {$\scriptscriptstyle 1$}; \node (c) at (.6,1) {$\scriptscriptstyle 2$};    \end{tikzpicture} 
 +   \begin{tikzpicture}[scale=.5,baseline=.1cm] \node (a) at (0,1) {$\scriptscriptstyle 1$}; \node (b) at (.8,1) {$\scriptscriptstyle 2$}; \draw (.2,1) --  (.6,1);  \end{tikzpicture}  \otimes 
      \begin{tikzpicture}[scale=.5,baseline=.1cm] \node (a) at (0,0) {$\scriptscriptstyle 1$}; \node (b) at (.8,0) {$\scriptscriptstyle 2$};  \end{tikzpicture}  
 +   \begin{tikzpicture}[scale=.5,baseline=.1cm] \node (a) at (0,0) {$\scriptscriptstyle 1$}; \node (c) at (.6,1) {$\scriptscriptstyle 2$};    \end{tikzpicture}  \otimes  
       \begin{tikzpicture}[scale=.5,baseline=.1cm] \node (a) at (0,1) {$\scriptscriptstyle 1$}; \node (c) at (.6,0) {$\scriptscriptstyle 2$}; \draw (a) --  (c);   \end{tikzpicture}  
 +   \begin{tikzpicture}[scale=.5,baseline=.1cm] \node (a) at (0,0) {$\scriptscriptstyle 1$}; \node (b) at (.8,0) {$\scriptscriptstyle 2$};  \end{tikzpicture}   \otimes  
       \begin{tikzpicture}[scale=.5,baseline=.1cm] \node (a) at (0,1) {$\scriptscriptstyle 1$}; \node (b) at (.8,1) {$\scriptscriptstyle 2$}; \draw (.2,1) --  (.6,1);  \end{tikzpicture}   
        $$
 \end{example}

Given a graph ${\bf g}\in {\mathcal G}_n$, we say that $1\le i\le n-1$ is a \emph{global split} of ${\bf g}$ if  
   $${\bf g} = {\bf g}|_{\{1,2,\ldots,i\}} \cup {\bf g}|_{\{i+1,\ldots,n\}}$$
  It is \emph{indecomposable} if it has no global splits. Hence $\mathcal G$ is a free algebra
  and that the generators are the graph ${\bf h}$ with no global splits. A complete list of free generators start with
    $${\scriptscriptstyle 1}, 
    \begin{tikzpicture}[scale=.5,baseline=-.05cm] \node (c) at (0,0) {$\scriptscriptstyle 1$}; \node (d) at (.6,0) {$\scriptscriptstyle 2$}; \draw (.15,0) --  (.4,0);  \end{tikzpicture},
\begin{tikzpicture}[scale=.5,baseline=.1cm]
	\node (a) at (0,0) {$\scriptscriptstyle 1$}; \node (b) at (-.5,1) {$\scriptscriptstyle 2$}; \node (c) at (.5,1) {$\scriptscriptstyle 3$};
	\draw (a) --  (b);  	\draw (b) --  (c);  \draw (a) --  (c);  \end{tikzpicture} ,
\begin{tikzpicture}[scale=.5,baseline=.1cm]
	\node (a) at (0,0) {$\scriptscriptstyle 1$}; \node (b) at (-.5,1) {$\scriptscriptstyle 2$}; \node (c) at (.5,1) {$\scriptscriptstyle 3$};
		\draw (b) --  (c);  \draw (a) --  (c);  \end{tikzpicture} ,
\begin{tikzpicture}[scale=.5,baseline=.1cm]
	\node (a) at (0,0) {$\scriptscriptstyle 1$}; \node (b) at (-.5,1) {$\scriptscriptstyle 2$}; \node (c) at (.5,1) {$\scriptscriptstyle 3$};
	\draw (a) --  (b);  	 \draw (a) --  (c);  \end{tikzpicture} ,
\begin{tikzpicture}[scale=.5,baseline=.1cm]
	\node (a) at (0,0) {$\scriptscriptstyle 1$}; \node (b) at (-.5,1) {$\scriptscriptstyle 2$}; \node (c) at (.5,1) {$\scriptscriptstyle 3$};
	\draw (a) --  (b);  	\draw (b) --  (c);  \end{tikzpicture} ,
\begin{tikzpicture}[scale=.5,baseline=.1cm]
	\node (a) at (0,0) {$\scriptscriptstyle 1$}; \node (b) at (-.5,1) {$\scriptscriptstyle 2$}; \node (c) at (.5,1) {$\scriptscriptstyle 3$};
	 \draw (a) --  (c);  \end{tikzpicture} ,
     \ldots $$
As before, a character $\zeta\colon{\mathcal P}\to\field$ is completely determined by its values on the free generators. The character $\zeta_1$ is defined by $\zeta_1(1)=1$ and set to zero for all other generators.
This defines $\zeta_1$ for all graphs as follows:
  $$\zeta_1({\bf g}) =
    \begin{cases}
      1&\text{if } {\bf g}=\emptyset\subseteq{[n]\choose 2}\,,\\
      0&\text{otherwise.}
    \end{cases}$$
Similarly, we define a character $\zeta_{\Gtwo}$ by letting $\zeta_{\Gtwo}(\Gtwo)=1$ and set $\zeta_{\Gtwo}$ to zero for all other generators.
This defines $\zeta_{\Gtwo}$ for all graphs as follows:
  $$\zeta_{\Gtwo}({\bf g}) =
    \begin{cases}
      1&\text{if }  {\bf g}=\big\{\{2i-1,2i\}: 1\le i\le n\big\} \,,\\
      0&\text{otherwise.}
    \end{cases}$$
Here $\zeta_{\Gtwo}({\bf g})=0$ unless ${\bf g}\in{\mathcal G}_{2n}$ a graph on an even number of entries.
We can define a character $\zeta_{{\bf h}}$ for ${\bf h}$ with no global splits which zero of all generators except ${\bf h}$ and $\zeta_{{\bf h}}({\bf h}) = 1$.

For any generator ${\bf h}$ we have a CHA given by $({\mathcal G},\zeta_{\bf h})$ and a Hopf morphism
$\Psi_{\bf h}\colon {\mathcal G}\to Sym$. 
There is a Hopf morphism $\mathcal P$ to $\mathcal G$. Given an order $P \in {\mathcal P}_n$ we construct a graph ${\bf g}_P\in {\mathcal G}_n$ 
where $\{i,j\}\in {\bf g}_P$ if and only if $i$ and $j$ are incomparable in $P$.
The linear extension of the map $P \mapsto {\bf g}_P$ defines a graded Hopf morphism ${\mathcal P}\rightarrow {\mathcal G}$. Indeed, it is straightforward to check that
  $${\bf g}_{m_{a,b}(P\otimes Q)} =m_{a,b}({\bf g}_P\otimes {\bf g}_Q)\qquad\text{and}\qquad {\bf g}_{\Delta_{a,b}(P)} =\Delta_{a,b}({\bf g}_P).$$
Now if we have an indecomposable order $Q$, it is not always the case that ${\bf g}_Q$ is an indecomposable graph. 
For example, ${\bf g}_{\!\!\! \begin{tikzpicture}[scale=.5,baseline=-.3cm] \node (c) at (0,0) {$\scriptscriptstyle 2$}; \node (d) at (0,.6) {$\scriptscriptstyle 1$}; \draw (0,.2) --  (0,.4);  \end{tikzpicture}\!\!\!}={\scriptscriptstyle 1\,\,2}$ is not indecomposable graph. When ${\bf g}_Q$ is indecomposable, the character $\zeta_Q$ and 
$\zeta_{{\bf g}_Q}$ are related as follows:
  $$ \zeta_Q(P)=\zeta_{{\bf g}_Q}({\bf g}_P).$$
Hence, the unicity of the morphism in Theorem~\ref{thm:ABS}  gives us that $\Psi_Q(P)=\Psi_{{\bf g}_Q}({\bf g}_P)$ in this case. We summarize this in the following Lemma.
\begin{lemma} Let $\psi\colon {\mathcal P}\rightarrow {\mathcal G}$ be the Hopf morphism defined by $\psi(P)={\bf g}_P$. For $Q$ an indecomposable poset of $\mathcal P$,
if $\psi(Q)$ is an indecomposable graph of $\mathcal G$, then
we have that
The morphisms $\Psi_Q\colon{\mathcal P}\to Sym$ and $\Psi_{\psi(P)}\colon{\mathcal G}\to Sym$ satisfy $\Psi_Q=\Psi_{\psi(Q)}\circ\psi$.
\end{lemma}
This shows that certain invariants $\Psi_Q$ are particular instances of a graph invariants. This is true for $\Psi_1$ and $\Psi_{\Ptwo}$.
Combining this with the previous subsection, this show that certain invariants on permutations are in fact graph invariants. 
In particular $\Psi_1$ is Stanley's chromatic symmetric function (refer to \cite{ABS} to see that $\Psi_1$ on graph is indeed Stanley's chromatic symmetric function).

\begin{remark} One may consider larger Hopf algebras containing $\mathcal G$ such as the Hopf algebra of hypergraphs.
The invariants we construct here would lift to natural invariants in these larger spaces. 
A (simple) hypegraphs $\bf h$ on $[n]$ is determined by a subset of $\{ A\subseteq [n] : |A|\ge 2\}$. We let ${\mathcal H}_n$ be the span of hypergraphs on $[n]$.
The Hopf algebra ${\mathcal H}=\bigoplus_{n\ge 0} {\mathcal H}_n$ is similar to $\mathcal G$ with product and coproduct using shifted union and standardized decomposition. We will sometimes be interested in the behavior of our result for $\mathcal H$.
\end{remark}

\subsection{The Combinatorial Hopf Algebra of isomorphism classes of graphs $\overline{\mathcal G}$}\label{ss:isograph}
In this section we consider the quotient $\mathcal G\to \overline{\mathcal G}$ where $\overline{\mathcal G}$ is the isomorphism classes of graphs.
As a graded vector space, $\overline{\mathcal G}=\bigoplus_{n\ge 0} \overline{\mathcal G}_n$ where $\overline{\mathcal G}_n$ is the span of the isomorphism classes
of graphs on the set $[n]=\{1,2\ldots, n\}$.

We represent $\overline{\bf g}\in \overline{\mathcal G}_n $ by drawing an unlabelled graph.
For instance, if $n=3$,
$$ 
\overline{\mathcal G}_3 = \field\Big\{ 
\begin{tikzpicture}[scale=.5,baseline=.1cm]
	\node (a) at (0,0) {$\scriptscriptstyle \bullet$}; \node (b) at (-.5,1) {$\scriptscriptstyle \bullet$}; \node (c) at (.5,1) {$\scriptscriptstyle \bullet$};
	\draw (a) --  (b);  	\draw (b) --  (c);  \draw (a) --  (c);  \end{tikzpicture} ,
\begin{tikzpicture}[scale=.5,baseline=.1cm]
	\node (a) at (0,0) {$\scriptscriptstyle \bullet$}; \node (b) at (-.5,1) {$\scriptscriptstyle \bullet$}; \node (c) at (.5,1) {$\scriptscriptstyle \bullet$};
		\draw (b) --  (c);  \draw (a) --  (c);  \end{tikzpicture} ,
\begin{tikzpicture}[scale=.5,baseline=.1cm]
	\node (a) at (0,0) {$\scriptscriptstyle \bullet$}; \node (b) at (-.5,1) {$\scriptscriptstyle \bullet$}; \node (c) at (.5,1) {$\scriptscriptstyle \bullet$};
	 \draw (a) --  (c);  \end{tikzpicture} ,
\begin{tikzpicture}[scale=.5,baseline=.1cm]
	\node (a) at (0,0) {$\scriptscriptstyle \bullet$}; \node (b) at (-.5,1) {$\scriptscriptstyle \bullet$}; \node (c) at (.5,1) {$\scriptscriptstyle \bullet$};
	 \end{tikzpicture} 
\Big\}\,.
$$
The operations on $\overline{\mathcal G}$ are induced by the operations on ${\mathcal G}$ by forgetting the labelling. The indecomposable elements now correspond
to connected graphs and the Hopf algebra $\overline{\mathcal G}$ is freely generated by connected graphs.
A complete list of indecomposable elements of $\overline{\mathcal G}$ start with
    $${\scriptscriptstyle \bullet}, 
    \begin{tikzpicture}[scale=.5,baseline=-.05cm] \node (c) at (0,0) {$\scriptscriptstyle \bullet$}; \node (d) at (.6,0) {$\scriptscriptstyle \bullet$}; \draw (.15,0) --  (.4,0);  \end{tikzpicture},
\begin{tikzpicture}[scale=.5,baseline=.1cm]
	\node (a) at (0,0) {$\scriptscriptstyle \bullet$}; \node (b) at (-.5,1) {$\scriptscriptstyle \bullet$}; \node (c) at (.5,1) {$\scriptscriptstyle \bullet$};
	\draw (a) --  (b);  	\draw (b) --  (c);  \draw (a) --  (c);  \end{tikzpicture} ,
\begin{tikzpicture}[scale=.5,baseline=.1cm]
	\node (a) at (0,0) {$\scriptscriptstyle \bullet$}; \node (b) at (-.5,1) {$\scriptscriptstyle \bullet$}; \node (c) at (.5,1) {$\scriptscriptstyle \bullet$};
		\draw (b) --  (c);  \draw (a) --  (c);  \end{tikzpicture} ,
     \ldots $$
The map $\Psi_\bullet\colon \overline{\mathcal G}\to Sym$ sends a graph to its chromatic symmetric function. From the previous section we have that
  $$\Psi_1(P)=\Psi_\bullet(\overline{\bf g}_{P}) \qquad\text{and}\qquad \Psi_1(\sigma)=\Psi_\bullet(\overline{\bf g}_{P_\sigma}),$$
for any $P\in{\mathcal P}_n$ and $\sigma\in{\mathcal S}_n$. 

\begin{remark}
To simplify the notation we will sometimes omit explicitly notating the quotient map $\mathcal{G} \to \overline{\mathcal{G}}$.
For example, we may simply write $\Psi_{\bullet}(\mathbf{g})$ for a labeled graph $\mathbf{g} \in \mathcal{G}$.
\end{remark}

One very striking conjecture of Stanley and Stembridge~\cite[Conjecture 5.5]{Stan-Stem} is the following.

\begin{conjecture}[\cite{Stan-Stem,Sta95}] If $P\in{\mathcal P}_n$ does not contains any $i,j,k,\ell\in [n]$ such that $i<_Pj<_Pk$ and $\ell$ is incomparable to $i,j,$ and $k$, then
the expansion of $\Psi_1(P)$ in terms of elementary symmetric functions $e_\lambda$ is positive.
\label{conj:Stan-Stem}
\end{conjecture}

A poset $P\in{\mathcal P}_n$ that contains some $i,j,k,\ell\in [n]$ such that $i<_Pj<_Pk$ and $\ell$ incomparable to $i,j,k$
is said to contain a $\bf 3+1$ pattern. If $P$ has a $\bf 3+1$ pattern, then $\overline{\bf g}_{P}$ will contain an induced claw; that is an induced subgraph of the form
$$\begin{tikzpicture}[scale=.5,baseline=.1cm]
	\node (a) at (0,0)  [inner sep =-2pt]{$\scriptscriptstyle \bullet$}; \node (b) at (1,0)  [inner sep =-2pt]{$\scriptscriptstyle \bullet$}; 
	\node (c) at (1,1)  [inner sep =-2pt]{$\scriptscriptstyle \bullet$}; \node (d) at (1,-1)  [inner sep =-2pt]{$\scriptscriptstyle \bullet$};
	\draw (a) --  (b);  	\draw  (a) -- (c);  \draw  (a) --  (d);  \end{tikzpicture}$$
On the other hand, we can find graph that has no claw, yet the chromatic symmetric function is not $e$-positive. 
For example,
$$\begin{tikzpicture}[scale=.5,baseline=.1cm]
	\node (a) at (0,0)  [inner sep =-2pt]{$\scriptscriptstyle \bullet$}; \node (b) at (1,0)  [inner sep =-2pt]{$\scriptscriptstyle \bullet$}; 
	\node (c) at (.5,.866)  [inner sep =-2pt]{$\scriptscriptstyle \bullet$}; \node (d) at (.5,1.732)  [inner sep =-2pt]{$\scriptscriptstyle \bullet$};
	\node (e) at (1.75,-.422)  [inner sep =-2pt]{$\scriptscriptstyle \bullet$}; \node (f) at (-.75,-.433)  [inner sep =-2pt]{$\scriptscriptstyle \bullet$}; 
	\draw (a) --  (b);  \draw (a) --  (f);  	\draw  (b) -- (c); \draw  (b) -- (e);  \draw  (a) --  (c);   \draw  (d) --  (c);  \end{tikzpicture}$$
is claw-free but not $e$-positive.
This shows that the conjecture above is really a property of posets rather than a property of graphs.
What we expect is true for graphs is the following conjecture of Stanley.
\begin{conjecture}[\cite{stanley98}] If ${\bf g}\in{\mathcal G}_n$ is claw free, then
the expansion of $\Psi_\bullet({\bf g})$ in terms of the Schur symmetric functions $s_\lambda$ is positive.
\end{conjecture}

In his original paper~\cite{Sta95}, Stanley gave many results and conjectures regarding  chromatic symmetric functions.
One interesting and useful result is the expansion of $\Psi_\bullet({\bf g})$ in terms of power sum symmetric functions $p_\lambda$.
Let $\omega$ be the (multiplicative) symmetric function involution that maps $\omega(p_k)=(-1)^{k-1}p_k$.
To state the result, let us recall the definition of the bond lattice of a graph.
Let ${\bf g}$ be a graph with vertex set $V$. 
A partition $Q=\{q_1,q_2,\dots,q_k\}$ of $V$ is said to be connected if  the restriction of ${\bf g}$ 
to each block $q_j$ is a connected graph. 
The {\em bond lattice} $L_{\bf g}$ is defined as the set of all connected partitions of ${\bf g}$,
partially ordered by refinement.
It is a ranked lattice with the rank of $Q$ given by $|V|-|Q|$ where $|Q|$ is the number of blocks in $Q$.
In the following statement $\hat{0}$ is the (unique) minimal element of $L_{\bf g}$ (
i.e. the partition of $V$ into singletons), $\mu$ is the M\"{o}bius function of $L_{\bf g}$, and $\lambda(Q)$ is the integer partition given by the sizes of the blocks in $Q$.

\begin{proposition}[\cite{Sta95}] \label{prop:omegap}
For any ${\bf g}\in{\mathcal G}_n$ we have
\begin{equation}
    \Psi_\bullet({\bf g}) = \sum_{Q \in L_{\bf g}} \mu(\hat{0}, Q) p_{\lambda(Q)}.
    \label{eq:csf_pn}
\end{equation}
Since the interval $[\hat{0}, Q]$ is a geometric lattice, the power sum expansion of $\omega(\Psi_\bullet({\bf g}))$ is positive.
\end{proposition}

\begin{definition}\label{def:h-alt}[positively h-alternating]
Given a homogeneous symmetric function $f$ of degree $n$, we say that $f$ is {\sl positively $h$-alternating} if the coefficient of the 
homogeneous symmetric function $(-1)^{n-\ell(\lambda)}h_{\lambda}$ is positive for every $\lambda\vdash n$. Here $\ell(\lambda)$ denotes the number of parts of $\lambda$.
We usually just say $h$-alternating and assume it must be positively alternating.
\end{definition}

We can use the Proposition~\ref{prop:omegap} to show that for any graphs, the chromatic symmetric function $\Psi_\bullet({\bf g})$ is $h$-alternating.

\begin{lemma}
Let $f$ be a homogeneous symmetric function of degree $n$.
If $\omega(f)$ is $p$-positive, then $f$ is $h$-alternating.
\label{lem:palt}
\end{lemma}
\begin{proof}
Recall that $\omega$ is an involution that maps $\omega(e_{\lambda}) = h_{\lambda}$.
Recall also that
\begin{equation}
p_n =  \det 
\begin{pmatrix}
e_1 &1 &0 &\cdots &0\\
2e_2 &e_1 &1  &\cdots &0\\
3e_3 &e_2 &e_1 &\cdots &0\\
\vdots &\vdots  & \ddots & \ddots & \vdots\\
ne_n & e_{n-1} & e_{n-2} & \cdots & e_1
\end{pmatrix}.
\label{eq:pe}
\end{equation}
Hence, $p_n$ is $e$-alternating in the sense that the coefficient of $(-1)^{n-\ell(\mu)}e_{\mu}$ is positive. 
Alternating is a multiplicative property, thus we then have that $p_{\lambda}$ is e-alternating. 
If $\omega(f)$ is $p$-positive, then no cancellation occur in $\omega(f)$ when expanded in the $e$ basis.
Hence, $\omega(f)$ is $e$-alternating. Applying $\omega$ again gives the desired result. 
\end{proof}

A direct consequence of this lemma and Proposition~\ref{prop:omegap} is the following.
\begin{corollary}
 For any graph ${\bf g}\in {\mathcal G}$, the chromatic symmetric function $\Psi_\bullet({\bf g})$ is $h$-alternating.\end{corollary}

\begin{remark}
Lemma~\ref{lem:palt} shows that the condition that $f$ is positively $h$-alternating can be thought of as a relaxation of $\omega(f)$ being $p$-positive.
We have shown above that for graphs, $\Psi_\bullet({\bf g})$ is $h$-alternating using Stanley's results.
But if one considers the Hopf algebra of hypergraphs (see~\cite{BB2017,Gru2016}) the situation is quite different.
The positivity result of Stanley is not true anymore for all hypergraphs $\bf h$ (see~\cite{stanley98}),
yet some $\Psi_\bullet({\bf h})$ are still $h$-alternating. Consider the simple example ${\bf h} = \{\{1,2,3\}, \{1,2,4,5\}\},$ the hypergraph on vertices $\{1,2,3,4,5\}$ with a two hyperedges $\{1,2,3\}$ and $\{1,2,4,5\}$.
In this case $\Psi_\bullet({\bf h}) = p_{1^5} - p_{3,1,1} - p_{4,1} +p_5$ and $\omega\big(\Psi_\bullet({\bf h})\big) = p_{1^5} - p_{3,1,1}+p_{4,1}+p_5$ is not $p$-positive. Yet $\Psi_\bullet({\bf h}) =2h_{1^5} - 6h_{2,1^3} + 7h_{2,2,1} + 6h_{3, 1, 1} -
5h_{3,2} - 9h_{4,1} + 5h_5$ is positively $h$-alternating. We will see in Theorem~\ref{thm:nablapos} why we care about $h$-alternating.
It appears difficult to find conditions on hypergraph to obtain an $h$-alternating symmetric function. 
We can convince ourselves that a necessary condition is that the edges of ${\bf h}$ are all of even size and all pairwise intersections are odd.
But this is not sufficient as seen with the example 
${\bf h} =  \{\{1,2,3,4\},\{4,5,6,7\},\{1,7,8,9\}\}$ and 
$\Psi_\bullet({\bf h}) = p_{1^9}-3p_{41^5}+3p_{7,1^2}-p_9$ which is not $\omega(p)$-positive nor $h$-alternating.
We propose the following condition on an hypergraph $\bf h$: for any edge $E\in {\bf h}$ and any contraction of the hypergraph ${\bf h}/S$ (with respect to a subset of edges $S\subseteq {\bf h}$), the cardinality of $E/S$ is even or 1.
It appears that in this case $\Psi_\bullet({\bf h})$ is $\omega(p)$-positive. This does not include the example above that is $h$-alternating but not $\omega(p)$-positive.
\end{remark}

Later we will study other kinds of invariants. In particular, the invariants defined by $\Psi_{21}\colon {\mathcal S}\to Sym$, $\Psi_{\Ptwo}\colon {\mathcal P}\to Sym$ and 
$\Psi_{\GGtwo}\colon \overline{\mathcal G}\to Sym$. Restricting ourselves to  permutations ${\mathcal S}$ gives us a small sampling of the invariants for graphs ${\mathcal G}$. 
However, it gives us enough of an idea along with quick verification by computer for understanding of some general phenomena.

\subsection{Chromatic polynomials} Before we start our study of new invariants, let us recall a few facts about the chromatic polynomial. Recall that for a graph ${\bf g}\in \mathcal G_n $ we have that $\chi_{{\bf g}}(t)=\phi_t\circ \Psi_\bullet({\bf g})$ is the classical chromatic polynomial of ${\bf g}$. In this paper we are also interested in the properties of other polynomials like 
$\chi_{\GGtwo,{\bf g}}(t)=\phi_t\circ \Psi_{\GGtwo}({\bf g})$. Let us recall some known properties of the classical chromatic polynomial $\chi_{{\bf g}}(t)$. Let
  $$\chi_{{\bf g}}(-t)=a_0 + a_1t +\cdots + a_nt^n$$
The following is a reformulation of an old theorem of Whitney:
\begin{proposition}[\cite{Whitney}] The polynomial $\chi_{{\bf g}}(-t)$ is positive, that is $a_i\ge 0$. Moreover $a_i>0$ for $i=0,1,...,n$ and it  counts  non-brocken circuits in the graph $\bf g$.
\label{prop:Whitney}
\end{proposition} 

More recently Huh has shown that the sequence of integers $a_0,a_1,\ldots, a_n$ forms a log-concave sequence.
\begin{theorem}[\cite{Huh}]  \label{prop:uni}The sequence of integers $a_0,a_1,\ldots, a_n$ is log-concave, and hence unimodal.
\end{theorem} 

It would be interesting to see if this is true for the other graph polynomial invariants we introduce here.
It is interesting to note that Adiprasito, Huh and Katz~\cite{AHKannals}
settled a long standing conjecture of Heron, Rota, and Welsh by extending the results of Theorem~\ref{prop:uni} to matroids.

\section{The invariant $\Psi_{21}\colon {\mathcal S}\to Sym$ and $\Psi_{\scriptscriptstyle \bullet - \bullet}\colon \overline{{\mathcal G}}\to Sym$ }\label{s:inv21}
We consider now the invariant $\Psi_{21}$.
We assume here that $n$ is an even integer.
For any permutation $\alpha$ of size $n$, $\Psi_{21}(\alpha)$ is defined as
\begin{equation}\label{eq:Pda-def}
\Pda=\sum_{a\models n}c_a^{21}(\alpha) M_a
\end{equation}
where the sum is taken over compositions $a$ of $n$ with even parts, 
and $c_a^{21}(\alpha)$ is the number of ways to decompose $\alpha$ 
into $21^*$-subsequences of type $a$, as defined in \cite{BB2017}:
$$
c_a^{21}(\alpha)=|\{A\models n\, : \text{for } 1\le i\le \ell,\  |A_i|=2a_i\ {\rm and}\ st(\alpha|A_i)=2143\dots(2a_i)(2a_i-1)\}| 
$$
when $a=(2a_1,\dots,2a_\ell)$. 

We now give a combinatorial result that relates the invariant $\Psi_{21}$ to Stanley's 
chromatic symmetric function $\Psi_1$. To this end, for any symmetric function $F(\mathbf{x})$ and positive integer $d$,
we denote by $F(\mathbf{x}^d)$ the symmetric function obtained by substituting each variable
$x_i$ of $\mathbf{x}$ by $x_i^d$.
Now, for any graph $\mathbf{g}$, any perfect matching $\pi$ of $\mathbf{g}$, we denote by $\mathbf{g}_{\downarrow\pi}$ 
the graph obtained by contracting the edges of $\pi$ in $\mathbf{g}$ (the order of contractions is of no importance).
\begin{proposition}\label{prop:contraction}
Let $n$ be an even integer and $\alpha$ a permutation of size $n$.
The symmetric function $\Pda$ satisfies
\begin{equation}\label{eq:Pda-prop}
\Pda=\sum_\pi \Psi_{\bullet}((\mathbf{g}_\alpha)_{\downarrow\pi})(\mathbf{x}^2)
\end{equation}
where $\mathbf{g}_\alpha$ is the incomparability graph of the poset $P_\alpha$ and the sum is taken over the set
of perfect matchings $\pi$ of $\mathbf{g}_\alpha$.
\end{proposition}
\begin{proof} 
Any decomposition of the permutation $\alpha$ as a shuffle of $21$ patterns is in bijection 
with a perfect matching of the graph $\mathbf{g}_\alpha$. 
For any perfect matching $\pi$ of $\mathbf{g}_\alpha$, the symmetric function that we compute 
is by definition exactly $\Psi_{\bullet}((\mathbf{g}_\alpha)_{\downarrow\pi})(\mathbf{x}^2)$.
\end{proof}

At the level of graphs, we have the same proof for the following proposition.

\begin{proposition}\label{prop:Gcontraction}
Let $n$ be an even integer and ${\bf g}\in \overline{\mathcal G}$ a graph.
The symmetric function $\Psi_{\GGtwo}({\bf g})$ satisfies
\begin{equation}\label{eq:Pda-prop}
\Psi_{\GGtwo}({\bf g})=\sum_\pi \Psi_{\bullet}({\bf g}_{\downarrow\pi})(\mathbf{x}^2)
\end{equation}
where the sum is taken over the set
of perfect matchings $\pi$ of $\bf g$.
\end{proposition}

This interpretation leads to the question of the $e$-positivity of $\Psi_{21}$.
Recall that the $e$-positivity conjecture is not well suited for graphs, and is not resolved for posets.
We are interested to exhibit a new phenomena here and will concentrate on permutations only.  
To be precise, since any variable $x_i$ appears in $\Pda$ with even exponents,
we may consider the symmetric function obtained by substituting each $x_i^2$ with $x_i$ 
in $\Pda$, develop the result in the basis of elementary symmetric functions,
and examine whether the coefficients are non-negative.
By abuse, we shall call this problem {\em $e$-positivity of $\Psi_{21}$}.
This question has a great similarity with original $\mathbf{3}+\mathbf{1}$ conjecture,
but we observe in our context new phenomena.

Of course if for any matching $\pi$, $(\mathbf{g}_\alpha)_{\downarrow\pi}$ is the incomparability 
graph of the poset avoiding $\mathbf{3}+\mathbf{1}$, then the $e$-positivity of $\Pda$ is predicted by Conjecture~\ref{conj:Stan-Stem}.
What is new in our context is the following: 
even if  $(\mathbf{g}_\alpha)_{\downarrow\pi}$ does not avoid the claw
\includegraphics[scale=0.2]{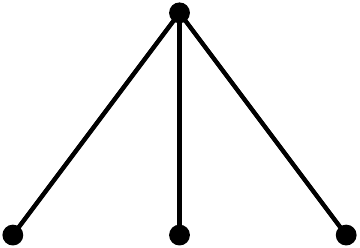}, $\Pda$ may be $e$-positive for a different reason.
This occurs when contracting matchings of the graph $(\mathbf{g}_\alpha)_{\downarrow\pi}$ present the claw pattern 
as well as the $4$-chain pattern \includegraphics[scale=0.3]{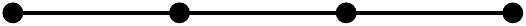}.
The reason is quite simple: the claw gives a term $-2e_{22}$, whereas the $4$-chain
gives a term $+2e_{22}$.

\begin{example} \label{ex:epos}
let $\alpha= 26153874$, we may draw the graph $G_\alpha$ as follow
  \begin{center}
    \begin{tikzpicture}
      \node  (A6) at (0,.2)  [inner sep =-2pt]  {$\scriptstyle\bullet$};
      \node  (A4) at (0,-.2) [inner sep =-2pt]  {$\scriptstyle\bullet$};
      \node  (A2) at (2,1.6) [inner sep =-2pt]  {$\scriptstyle\bullet$};
      \node  (A1) at (2,1.2) [inner sep =-2pt]  {$\scriptstyle\bullet$};
      \node  (A5) at (2,.2) [inner sep =-2pt]  {$\scriptstyle\bullet$};
      \node  (A3) at (2,-.2) [inner sep =-2pt]  {$\scriptstyle\bullet$};
      \node  (A8) at (2,-1.2) [inner sep =-2pt]  {$\scriptstyle\bullet$};
      \node  (A7) at (2,-1.6) [inner sep =-2pt]  {$\scriptstyle\bullet$};
      \node[right = 1pt of A1] {$\scriptscriptstyle 1$};
      \node[right = 1pt of A2] {$\scriptscriptstyle 2$};
      \node[right = 1pt of A3] {$\scriptscriptstyle 3$};
      \node[left = 1pt of A4] {$\scriptscriptstyle 4$};
      \node[right = 1pt of A5] {$\scriptscriptstyle 5$};
      \node[left = 1pt of A6] {$\scriptscriptstyle 6$};
      \node[right = 1pt of A7] {$\scriptscriptstyle 7$};
      \node[right = 1pt of A8] {$\scriptscriptstyle 8$};
      \draw[color=red, densely dotted,thick] (A6) -- (A4);
      \draw[color=red, densely dotted,thick] (A2) -- (A1);
      \draw[color=red, densely dotted,thick] (A5) -- (A3);
      \draw[color=red, densely dotted,thick] (A8) -- (A7);
      \draw[thick] (A6) -- (A1);
      \draw[color=blue,  dashed,thick] (A6) -- (A3);
      \draw[thick] (A6) -- (A5);
      \draw[color=blue,  dashed,thick] (A5) -- (A4);
      \draw[thick] (A8) -- (A4);
     \end{tikzpicture}
  \end{center}
  For this graph, there are two possible perfect matchings: $\pi_1=\{\red{21},\red{53},\red{64},\red{87}\}$ indicated in dotted red in $\mathbf{g}_\alpha$,
  or we can replace $\red{53},\red{64}$ by $\blue{54},\blue{63}$ and have $\pi_2=\{\red{21},\blue{54},\blue{63},\red{87}\}$. The graph 
  $(\mathbf{g}_\alpha)_{\downarrow\pi_1}$ is a claw and $(\mathbf{g}_\alpha)_{\downarrow\pi_2}$ is a 4-chain. The sum of the two chromatic symmetric functions gives us 
  $e_{211} + 7e_{31} + 8e_4$.
\end{example}

A computer exploration gives us  exactly 32 permutations $\alpha$ of size $n=8$ which give a claw after contracting a perfect matching, but also still exhibit unexpected $e$-positivity of $\Pda$ like shown in Example~\ref{ex:epos}. These lie outside the conjectural $e$-positivity for $\mathbf{3}+\mathbf{1}$ avoiding posets, and we find it interesting that this larger family of posets coming from permutations has an $e$-positive invariant from a CHA.
Since there is no known interest for such study at this point, we will leave further exploration of this to the interested reader.

We now turn to a result analogous to Proposition \ref{prop:omegap}.
Let ${\bf g}\in \overline{\mathcal G}$ be a graph with vertex set $V$.
We assume $|V|$ is even because otherwise the invariant is equal to zero.
We define the \emph{($\GGtwo$)-bond poset} $L_{\bf g}^{\GGtwo}$ of ${\bf g}$. 
A set partition $K$ of $V$ is in $L_{\bf g}^{\GGtwo}$ if and only if $K$ coarsens a partition of $V$ given by a perfect matching
and the graph restricted on each block of $K$ is connected. 
The poset $L_{\bf g}^{\GGtwo}$ is ordered by refinement, and its minimal elements are set partitions induced by perfect matchings
of $\mathbf{g}$.
Next we define the function $\nu_{\bf g}^{\GGtwo}$ on $L_{\bf g}^{\GGtwo}$ as follows.
For any minimal element of $\pi \in L_{\bf g}^{\GGtwo}$ we set 
$\nu_{\bf g}^{\GGtwo}(\pi)=1$. 
For any other element $K\in L_{\bf g}^{\GGtwo}$, we define $\nu_{\bf g}^{\GGtwo}(K)$ by the condition that 
$$\sum_{K'\le K} \nu_{\bf g}^{\GGtwo}(K')=0$$
where the sum is other the elements smaller or equal to $K$ in $L_{\bf g}^{\GGtwo}$.
Now for $K\in L_{\bf g}^{\GGtwo}$, let us define $\lambda(K)$
as the integer partition given by the sizes of the blocks of $K$.
\begin{proposition} \label{prop:bond-21}
For any ${\bf g}\in{\overline{\mathcal G}}_n$, we have
\begin{equation}
   \Psi_{\GGtwo}({\bf g}) = \sum_{K \in L_{\bf g}^{\GGtwo}} \nu_{\bf g}^{\GGtwo}(K)\, p_{\lambda(K)}.
    \label{eq:csf}
\end{equation}
\end{proposition}

\begin{proof}
Equation \ref{eq:Pda-prop} implies that the ($\GGtwo$)-invariant $\Psi_{\GGtwo}({\bf g})$ is the sum of 
$\Psi_\bullet({\bf g}_{\downarrow\pi})$ over the set of perfect matchings $\pi$ of ${\bf g}$ 
(with suitable renormalizations). 
We apply Proposition \ref{prop:omegap} and get that for any perfect matching $\pi$,
$\Psi_\bullet({\bf g}_{\downarrow\pi})$ is the sum of $p_{\lambda(Q)}(\mathbf{x}^2)$, weighted by the M\"{o}bius function associated to the usual bond lattice of ${\bf g}_{\downarrow\pi}$.
Next we observe that in the ($\GGtwo$)-bond poset, the value of $\nu_{\bf g}^{\GGtwo}(K)$ 
is the sum of the values of $\mu(\hat{0}, Q)$ with $Q$ obtained as a contraction of $Q=K_{\downarrow\pi}$
with respect to each perfect matching $\pi$.
To conclude, we just use the fact that $p_k(\mathbf{x}^2)=p_{2k}(\mathbf{x})$.
\end{proof}

\begin{example}
Let ${\bf g}=$\lower 5pt\hbox{\includegraphics[scale=0.7]{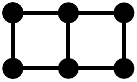}}.
The poset $L_{\bf g}^{\GGtwo}$ is given in Figure \ref{fig:bond-ex}, with the values of $\nu_{\bf g}^{\GGtwo}$ in red beneath each element of the poset.
\begin{figure}[htbp]
\begin{center}
\includegraphics[scale=0.3]{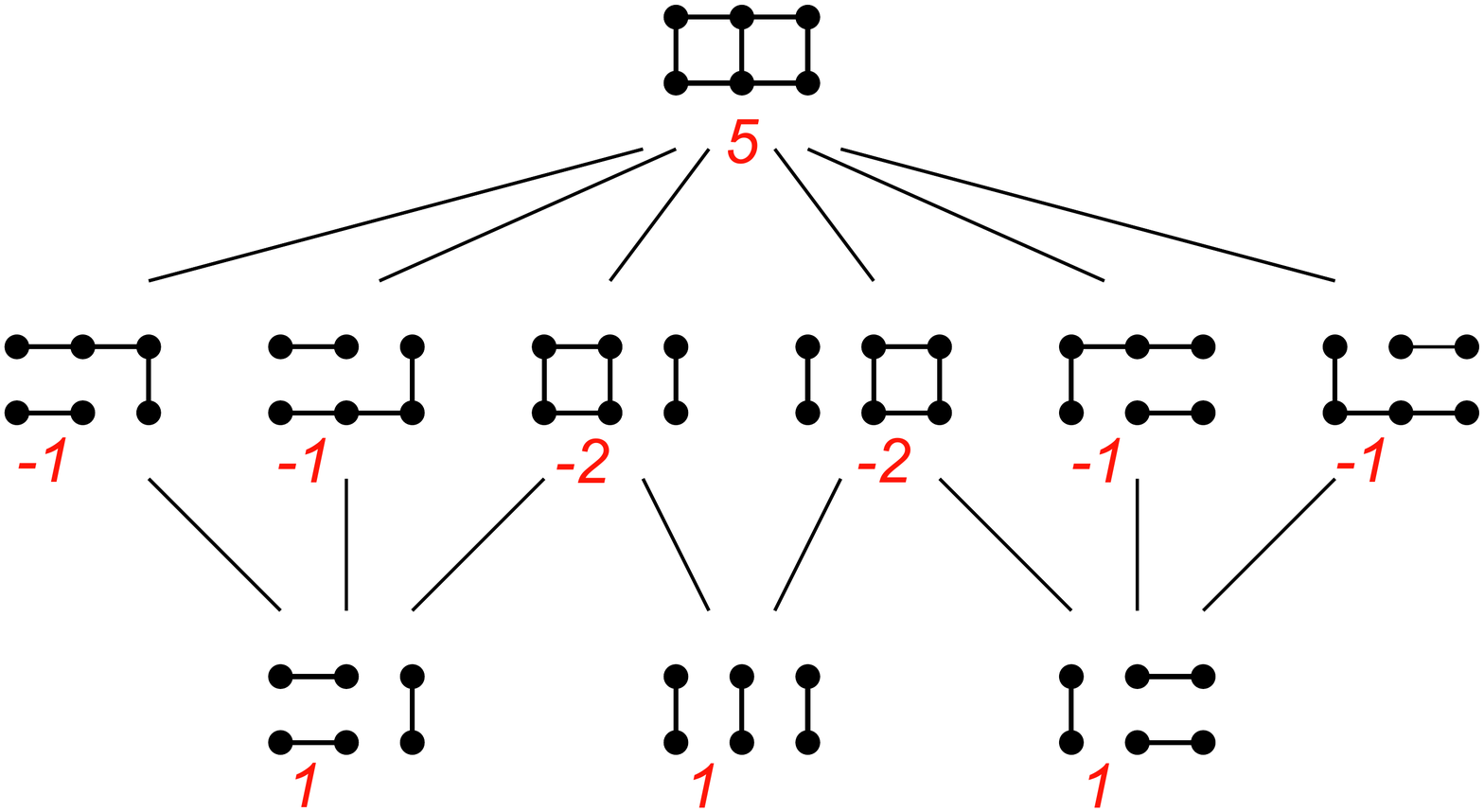}
\end{center}
\caption{Example of ($\scriptstyle \bullet-\bullet$)-bond poset}
\label{fig:bond-ex}
\end{figure}
We get the expression for the invariant $\Psi_{\GGtwo}({\bf g})$: $3p_2^3-8p_4p_2+5p_6$.
\end{example}

\section{The $h$-basis and the operator $\nabla$}\label{s:nabla}
In this section we return to the more familiar case of the usual chromatic symmetric symmetric function of explore some consequences of the $h$-alternating property.
For every positive integer $a$ we define an operator $C_a$ on $Sym$ by
\[C_a F[X] := \left(\left(\frac{-1}{q}\right)^{a-1}F\left[\frac{1-1/q}{z}\right] \sum_{m \geq 0}z^m h_m[X]\right)\bigg\rvert_{z^a}\]
for any $F \in Sym$.
For any integer composition $\alpha$ we define
\[C_{\alpha} := C_{\alpha_1} \circ C_{\alpha_2} \circ \cdots \circ C_{\alpha_{\ell}}\]
to be the compositions of operators.
The operators $C_{\alpha}$ were defined by Haglund, Morse, and Zabrocki~\cite{HMZ12} in the context of the compositional shuffle conjecture, generalizing a property discovered in~\cite{BDZ10}.
One identity involving  the operators $C_{\alpha}$ is that
\[\sum_{\alpha \vDash n} C_{\alpha} 1 = e_n\]
which implies one obtains a refinement the shuffle conjecture of  Haglund, Haiman, Loehr, Remmel, and Ulyanov~\cite{HHLRU05} when applying $\nabla$ to $C_{\alpha} 1$.
The compositional shuffle conjecture is now a theorem due to Carlsson and Mellit~\cite{CM18}.

\begin{theorem}[The compositional shuffle theorem \cite{CM18}]
For any composition $\alpha$, 
\[\nabla C_{\alpha} 1 = \sum_{\substack{PF \in \PF_n\\ \comp(PF) = \alpha}} t^{\area(PF)} q^{\dinv(PF)} F_{\ides(PF)}.\]
\label{thm:shuffle}
\end{theorem}

We briefly recall the following definitions related to Theorem~\ref{thm:shuffle}. The operator $\nabla$ was introduced in~\cite{BGHT99} 
as the linear operator on symmetric functions with eigenvalues given by the Macdonald symmetric functions and prescribed eigenvalues
(see~\cite{BGHT99} for more details).
A Dyck path $D$ of size $n$ is a lattice path from $(0,0)$ to $(n,n)$ with steps $(0,1)$ and $(1,0)$ which remains above the diagonal.
We may encode its contact points with the diagonal through a composition of $n$, denoted $\comp(D)$.
The area of a Dyck path $D$, denoted $\area(D)$, is the number of $1\times1$ squares between $D$ and the diagonal.
The type of a Dyck path $D$, denoted $\type(D)$, is the partition of $n$ with parts corresponding the size of the vertical rises.
A parking function $PF$ (with underlying path $D$) is a labelling of the $(0,1)$ steps of a Dyck path $D$
such that the labels associated to adjacent $(0,1)$ steps are increasing.
We denote the set of all parking functions of size $n$ by $\PF_n$.
We extend the statistics $\area$ and $\comp$ to parking functions; we do not recall here the definition of $\dinv$
as we shall put $q=1$ in the sequel.
To any parking function $PF$ we associate a word $w(PF)$  obtained by reading the labels from highest 
to lowest diagonal and right to left within each diagonal. It is clear that $w(PF)$ is a permutation of size $n$,
and the $\ides$ statistic is defined as the composition whose partial sums give the descent set of $w(PF)^{-1}$.
Figure~\ref{fig:defPF} illustrates these definitions.
In this example, one has: $\comp(PF)=(2,1,5)$, $\area(PF)=8$, $w(PF)=74865123$ hence $\ides(PF)=(3,2,1,2)$.
The underlying Dyck path $D$ in Figure~\ref{fig:defPF} has $\type(D)=(3,2,1,1,1)\vdash 8$.
\begin{figure}[htbp]
\begin{center}

\includegraphics[scale=0.25]{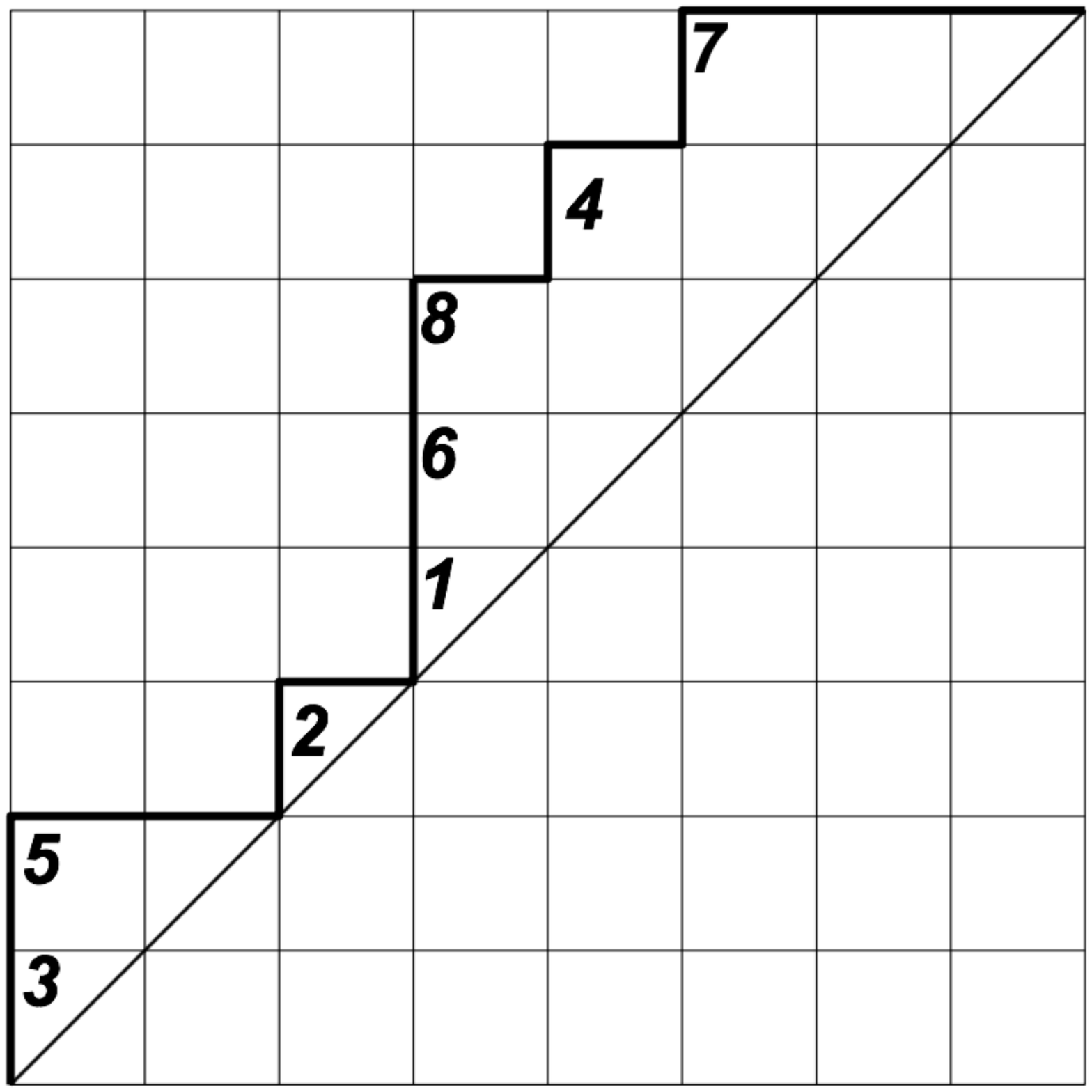}
\end{center}
\caption{A parking function of size 8.}
\label{fig:defPF}
\end{figure}


The shuffle theorem gives a combinatorial interpretation of $\nabla e_n$.
A natural next step is to look for an understanding of $\nabla F$ given another symmetric function $F$.
In this section we initiate the study of $\nabla \Psi_\bullet({\bf g})$ for any graph ${\bf g}$.
The quantity $\nabla \Psi_\bullet({\bf g})$ is a generalization of $\nabla e_n$ as $\Psi_\bullet({\bf g}) = n! e_n$ when ${\bf g}$ is the complete graph on $n$ vertices.
Other symmetric functions which have been considered include a conjectured formula for $\nabla s_{\lambda}$~\cite{nablaSchur} as well as a conjectured formula for $\nabla p_n$~\cite{sqPathsConj} which has been proven~\cite{sqPaths}.
Work has also been done on $\nabla m_{\lambda}$~\cite{Sergel}.

\begin{example}
Let ${\bf g}$ be the path graph on $3$ vertices which is the smallest connected graph which is not a complete graph.
We then have
\begin{align*}
\Psi_\bullet({\bf g}) &= p_{1^3} - 2p_{2,1} + p_{3}\\
&= 4h_{1^3} - 7h_{2, 1} + 3h_{3}
\end{align*}
as the chromatic symmetric function.
Applying $\nabla$ we find
\[\nabla \Psi_\bullet({\bf g}) = (-q^3t-q^2t^2-qt^3+4q^3+4q^2t+4qt^2+4t^3+4qt)s_{1^3}\]
\[ \qquad\qquad+ (-q^2t-qt^2+4q^2+4qt+4t^2+4q+4t)s_{2, 1} + 4s_3\]
and specializing to $q=1$ we obtain
\[(\nabla \Psi_\bullet({\bf g}))\big\rvert_{q=1} = (3t^3+3t^2+7t+4)s_{1^3} + (3t^2+7t+8)s_{2, 1} + 4s_3\]
which has coefficients that are polynomials in $t$ with nonnegative coefficients when expanded in the basis of Schur functions.
We will see this Schur positivity at $q=1$ is always the case 
and will attempt to understand this Schur positive expansion through the compositional shuffle theorem by understanding the $h$-basis expansion of $\Psi_\bullet({\bf g})$.
Furthermore, we also find that we have $e$-positivity which is demonstrated by
\[ (\nabla \Psi_\bullet({\bf g}))\big\rvert_{q=1}= 4e_{1, 1, 1} + (3t^2+7t)e_{2, 1} + 3t^3e_3\]
in this example.
\end{example}

Observing that when $q=1$ the plethystic shift vanishes, it follows that 
\begin{equation}
    \left(C_{\alpha} 1 \right)\big\rvert_{q=1} = (-1)^{|\alpha| - \ell(\alpha)}h_{\alpha}
    \label{eq:h}
\end{equation}
and we see a connection between the $C$-operators and the $h$-basis.
This observation is also made in~\cite[Lemma 2.1]{Sergel}.
Notice that this implies that  
\begin{equation}
    (C_{\alpha} 1)\big\rvert_{q=1} (C_{\beta} 1)\big\rvert_{q=1} = (C_{\alpha\beta} 1)\big\rvert_{q=1}
    \label{eq:Cmult}
\end{equation}
and so the $C$-operators are multiplicative when $q=1$.
Garsia and Sergel have shown the equation
\begin{equation}
    (-1)^{n-1} p_n = \sum_{\alpha \vDash n} [\alpha_1]_q C_{\alpha} 1 =  \sum_{\alpha \vDash n} \alpha_1  (C_{\alpha} 1)\big\rvert_{q=1}
    \label{eq:pn}
\end{equation}
which can be found in Sergel's thesis~\cite[Theorem 4.1.4]{SergelThesis}. The last equality follows from the fact that $p_{\lambda} = (p_{\lambda})\big\rvert_{q=1}$.
This can be used to give another proof of the fact in Lemma~\ref{lem:palt} that if $\omega(F)$ is $p$-positive, then $F$ is $h$-alternating.

\begin{lemma}
If $F \in Sym_n$ is such that $\omega(F)$ is $p$-positive, then 
\[F = \sum_{\alpha \vDash n} b_{\alpha} (C_{\alpha} 1)\big\rvert_{q=1} = \sum_{\alpha \vDash n} (-1)^{n - \ell(\alpha)}b_{\alpha} h_{\alpha} \]
for some $b_{\alpha} \geq 0$ for all $\alpha \vDash n$.
\label{lem:Cq=1}
\end{lemma}
\begin{proof}
The second equality in the proposition follows immediately from Equation~(\ref{eq:h}).
So, we must only demonstrate the positivity of the $b_\alpha$ in the first equality.
Consider an integer partition $\lambda \vdash n$, then
\begin{align*}
    p_{\lambda} &= p_{\lambda_1}p_{\lambda_2}\cdots p_{\lambda_{\ell}}\\
    &= (-1)^{n - \ell}\left(\sum_{\alpha \vDash \lambda_1} \alpha_1 (C_{\alpha} 1)\big\rvert_{q=1}\right) \left(\sum_{\alpha \vDash \lambda_2} \alpha_1 (C_{\alpha} 1)\big\rvert_{q=1}\right) \cdots \left(\sum_{\alpha \vDash \lambda_{\ell}} \alpha_1 (C_{\alpha} 1)\big\rvert_{q=1}\right)
\end{align*}
where we have used Equation~(\ref{eq:pn}).
We then see that
\[\omega(p_{\lambda}) = \left(\sum_{\alpha \vDash \lambda_1} \alpha_1 (C_{\alpha} 1)\big\rvert_{q=1}\right) \left(\sum_{\alpha \vDash \lambda_2} \alpha_1 (C_{\alpha} 1)\big\rvert_{q=1}\right) \cdots \left(\sum_{\alpha \vDash \lambda_{\ell}} \alpha_1 (C_{\alpha} 1)\big\rvert_{q=1}\right)\]
and the existence of positive  $b_\alpha$ in the proposition follows by Equation~(\ref{eq:Cmult}).
\end{proof}

In~\cite{BGHT99} many positivity conjectures are given regarding $\nabla$, but at $q=1$ most can be shown.
In particular, in~\cite[Theorem 4.4]{Lenart2000},  Lenart shows that $\nabla s_{\lambda/\mu}\big\rvert_{q=1}$ is Schur positive (up to a global sign).  
This gives that $(-1)^{n-1}\nabla (h_n)\big\rvert_{q=1}$ is Schur positive. Since $\nabla$ at $q=1$ is multiplicative, we have that $\nabla$ applied to any 
(positively) $h$-alternating symmetric function expand positively in term of Schur function. We can push the argument of Lenart further and in fact show that $\nabla (s_{\lambda/\mu})\big\rvert_{q=1}$ is $e$-positive (up to a global sign). Since we only need $(-1)^{n-1}\nabla (h_n)\big\rvert_{q=1}$, we will show the following lemma.
\begin{lemma}\label{lem:epos} Let ${\mathcal D}_n$ denote the set of Dyck paths of size $n$ and let $1\le k\le n$. 
For $D\in {\mathcal D}_n$ let  $\comp_1(D)$ be the first part of $\comp(D)$.
We have
 $$\nabla (s_{k1^{n-k}})\big\rvert_{q=1} = (-1)^{k-1} \sum_{D\in {\mathcal D}_n \atop \comp_1(D)\ge k } t^{\area(D)} e_{\type(D)}\,.$$
\end{lemma} 
\begin{proof}
We proceed by induction on $k$. For $k=1$, we have $s_{1^{n}}=e_n$ and in this case it is well known (see~\cite{HHLRU05}) that
  $$\nabla (e_n)\big\rvert_{q=1} = \sum_{D\in {\mathcal D}_n } t^{\area(D)} e_{\type(D)}\,.$$
For $k\ge 2$, we have the identity $s_{k1^{n-k}}=h_{k-1}e_{n-k+1}-s_{(k-1)1^{n-k+1}}$. Applying $\nabla$ at $q=1$ to both sides, and using the induction hypothesis, we obtain
  $$ (-1)^{k-1}(\nabla s_{k1^{n-k}})\big\rvert_{q=1} =(-1)^{k-1}\Big(  \nabla (h_{k-1}e_{n-k+1})\big\rvert_{q=1}  -  \nabla (s_{(k-1)1^{n-k+1}})\big\rvert_{q=1}\Big)
  $$
  $$ = (-1)^{k-2}\Big(  \nabla (s_{(k-1)1^{n-k+1}})\big\rvert_{q=1} -     \nabla (h_{k-1})\big\rvert_{q=1} \nabla( e_{n-k+1})\big\rvert_{q=1} \Big)
  $$
 $$ =\Big(\hskip-10pt \sum_{D\in {\mathcal D}_n \atop \comp_1(D)\ge k-1 } t^{\area(D)} e_{\type(D)}\Big) -
  \Big(\hskip-10pt\sum_{D'\in {\mathcal D}_{k-1} \atop \comp_1(D')\ge k-1 } t^{\area(D')} e_{\type(D')}\Big)
  \Big(\hskip-10pt\sum_{D''\in {\mathcal D}_{n-k+1} } t^{\area(D'')} e_{\type(D'')}\Big)  $$
For $D'\in {\mathcal D}_{k-1}$ if $\comp_1(D')\ge k-1$, then we must have $\comp(D')=k-1$.
Next for $D''\in {\mathcal D}$, the concatenation $D=D'D'' \in {\mathcal D}_n$ is such that $\comp_1(D)= k-1$.
Hence, the difference above gives us exactly 
$$\Big(\hskip-10pt \sum_{D\in {\mathcal D}_n \atop \comp_1(D)\ge k } t^{\area(D)} e_{\type(D)}\Big)$$
as the only Dyck paths $D$ with $\comp_1(D)\ge k$ remain. This concludes the proof.
\end{proof}

Using the lemma above, we thus have the following theorem.

\begin{theorem}\label{thm:nablapos}
If $F \in Sym_n$ is positively h-alternating (see Def.~\ref{def:h-alt}), then $(\nabla F)\big\rvert_{q=1}$ is $e$-positive (and Schur positive). In particular, if  
$\omega(F)$ is $p$-positive, then $(\nabla F)\big\rvert_{q=1}$ is $e$-positive. 
\end{theorem}
 
By combining Proposition~\ref{prop:omegap} and Theorem~\ref{thm:nablapos} we obtain the following result.

\begin{corollary}
If ${\bf g}$ is a graph on $n$ vertices, then
\[(\nabla \Psi_\bullet({\bf g}))\big\rvert_{q=1} = \sum_{\lambda \vdash n} d_{\lambda}(t) s_{\lambda}= \sum_{\lambda \vdash n} d'_{\lambda}(t) e_{\lambda}\]
where $d_{\lambda}(t),d'_{\lambda}(t) \in \mathbb{N}[t]$ for all $\lambda \vdash n$.
\label{cor:csf}
\end{corollary}

Next proposition gives informations about the polynomial coefficients $d_{\lambda}(t) \in \mathbb{N}[t]$.
To state its third assertion, we need the following definitions.
\begin{definition}
Let $\lambda=(\lambda_1,\dots,\lambda_\ell)$ be a partition.
We define ${\mathcal D}_\lambda$ as the set of Dyck paths $D$ such that  $\comp(D)$ refines $\lambda$:
we may write 
$$\comp(D)=(c_{1,1},\dots,c_{1,k_1},c_{2,1},\dots,c_{2,k_2},\dots,c_{\ell,1},\dots,c_{\ell,k_\ell})$$
with $c_{j,1}+\dots+c_{j,k_j}=\lambda_j$ for any $j\in\{1,\dots,\ell\}$.
Using these notations, we define the constant $a_{comp(D),\lambda}$ as $\prod_{j=1}^\ell c_{j,1}$.
\end{definition}
\begin{proposition}\label{prop:info-dlambda}
We have:
\begin{align*}
d_\lambda(0) &= a({\bf g})f^\lambda;\\
d_{(n)}(t) &= d_{(n)}(0)=a({\bf g});\\
d_{1^n}(t) &= \sum_{Q\in L_{\bf g}} \sum_{D\in {\mathcal D}_{\lambda(Q)}} a_{comp(D),\lambda}\, |\mu(0,\pi)| \, t^{\area(D)},
\end{align*}
where $f^\lambda$ as usual denotes the number of standard Young tableaux of shape $\lambda$.
\end{proposition}

\begin{proof}
We use \eqref{eq:csf_pn}, \eqref{eq:pn} and the fact that $\nabla\left( \cdot \right) \big\rvert_{q=1}$ is multiplicative, to write:
\begin{align*}
\Psi_\bullet({\bf g})
&= \sum_{Q \in L_{\bf g}} \mu(\hat{0}, Q) p_{\lambda(Q)}\\
&= \sum_{Q \in L_{\bf g}} |\mu(\hat{0}, Q)| \prod_{i=1}^{\ell} \Big(\sum_{\alpha \vDash \lambda_i} \alpha_1  (C_{\alpha} 1)\big\rvert_{q=1}  \Big)
\end{align*}
where for each $Q$, we denote $\lambda(Q)=(\lambda_1,\dots,\lambda_\ell)$.
We then apply Theorem~\ref{thm:shuffle} and get:
\begin{equation}\label{eq:sun}
(\nabla \Psi_\bullet({\bf g}))\big\rvert_{q=1} = 
\sum_{Q \in L_{\bf g}} |\mu(\hat{0}, Q)| \prod_{i=1}^{\ell} 
\Big(
\sum_{\alpha \vDash \lambda_i} \alpha_1 
\left(\sum_{\substack{PF \in \PF_{\lambda_i}\\ \comp(PF) = \alpha}} t^{\area(PF)} F_{\ides(PF)}\right)
 \Big)
\end{equation}
which is the key expression to derive informations about the coefficients $d_\lambda(t)$.

To compute $d_\lambda(0)$, we set $t=0$ in \eqref{eq:sun}
whence we keep only the parking functions with area equal to zero. 
Thus we obtain
\begin{align*}
\sum_{\lambda\vdash n} d_\lambda(0) s_\lambda &=  
\sum_{Q \in L_{\bf g}} |\mu(\hat{0}, Q)| \prod_{i=1}^{\ell} 
\Big(
\sum_{\substack{PF \in \PF_{\lambda_i}\\ \comp(PF) = (1^{\lambda_i})}}  F_{\ides(PF)}
 \Big)\\
 &=\sum_{Q \in L_{\bf g}} |\mu(\hat{0}, Q)|
  \prod_{i=1}^{\ell} 
\Big(
h_1^{\lambda_i}
 \Big)\\
 & = a({\bf g}) h_1^n = a({\bf g}) \sum_{\lambda\vdash n} f^\lambda s_\lambda,
 \end{align*}
where we have used Whitney's result \cite{Whitney} which says $\sum_{Q \in L_{\bf g}} |\mu(\hat{0}, Q)|=a({\bf g})$.
Whence the first assertion of the proposition.

Next, we consider the coefficient of $F_{(n)}=s_{(n)}$ in \eqref{eq:sun}.
We recall that the multiplicative rule for the fundamental quasisymmetric functions (see for example \cite{Gessel}) 
implies that for two compositions $\alpha\vDash a$ and $\beta\vDash b$ the coefficient of $F_{(a+b)}$
in $F_\alpha.F_\beta$ is $1$ if $\alpha=(a)$ and $\beta=(b)$ and $0$ if not.
This leads to the fact that for any $\lambda_i$, exactly one $PF \in \PF_{\lambda_i}$ may contribute 
(the one with zero area, with the labels put in decreasing order), thus
$$d_{(n)}(t) = \sum_{Q \in L_{\bf g}} |\mu(\hat{0}, Q)| F_{(n)},$$
and we get $d_{(n)}(t)=d_{(n)}(0)=a({\bf g})$.

We now turn to the last assertion of the proposition.
Let us consider the coefficient of $F_{(1^n)}=s_{(1^n)}$ in \eqref{eq:sun}.
In the present case, the multiplicative rule for the fundamental quasisymmetric functions
implies that for two compositions $\alpha\vDash a$ and $\beta\vDash b$ the coefficient of $F_{(1^{a+b})}$
in $F_\alpha.F_\beta$ is $1$ if $\alpha=(1^a)$ and $\beta=(1^b)$ and $0$ if not.
Now, we observe that for any Dyck path $D$ of size $k$, there is exactly one parking function $PF$ with underlying 
Dyck path $D$ such that $\ides(PF)=(1^k)$.
Putting all this together, we obtain:
\begin{align*}
d_{1^n}(t) &= \sum_{Q\in L_{\bf g}} |\mu(\hat{0}, Q)|
\prod_{i=1}^{\ell} 
\Big(
\sum_{\alpha \vDash \lambda_i} \alpha_1 
(\sum_{\comp(D) = \alpha} t^{\area(D)})
\Big)\\
&= \sum_{Q\in L_{\bf g}} |\mu(\hat{0}, Q)| 
\sum_{D\in {\mathcal D}_{\lambda(Q)}} a_{\comp(D),\lambda}\, t^{\area(D)}).
\end{align*}
\end{proof}

\section{Some general properties of invariants}\label{s:gen}
We have demonstrated how one can obtain many symmetric function and polynomial invariants from CHAs by considering various characters on Hopf algebras of permutations, posets, and graphs.
In previous sections we have analyzed some particular choices of characters.
In this final section we will discuss some properties which hold generally for many choices of characters on $\overline{\mathcal G}$.
We can obtain a different invariant for any collection of connected graphs.
When all connected graphs in the chosen collection have the same number of vertices, we find the symmetric function to have nice properties including $\omega(p)$-positivity (up to a global sign) as well as a cancellation free combinatorial reciprocity result.
Furthermore, we show how these invariants fit into the framework of \emph{scheduling problems}~\cite{sched}.

Let $A$ be any collection of connected graphs.
We then define a character $\zeta_A$ such that for a connected graph $\mathbf{g}$
\[\zeta_A(\mathbf{g}) = 
\begin{cases} 1  & \mathbf{g} \in A\\
0 & \text{otherwise}
\end{cases}\]
and then extend to the whole Hopf algebra.
This can be done since connected graphs generate $\overline{\mathcal{G}}$ as an algebra.
The character $\zeta_A$ is called \emph{homogeneous of degree $d$} if each graph in $A$ has exactly $d$ vertices.
Given any graph $\mathbf{g}$ let $\Pi_A(\mathbf{g})$ denote all set partitions $\pi = \pi_1 / \pi_2 / \cdots / \pi_{\ell}$ of the vertex set of $\mathbf{g}$ such that $\mathbf{g}|_{\pi_i} \in A$ for all $1 \leq i \leq \ell$.
For any graph $\mathbf{g}$ and any $\pi \in \Pi_A(\mathbf{g})$ we denote by $\mathbf{g}_{\downarrow\pi}$ the graph obtained by from $\mathbf{g}$ contracting the edges of $\mathbf{g}|_{\pi_i}$ to a single vertex for each $1 \leq i \leq \ell$.

The next lemma follows from the definition of the morphism $\Psi_{\zeta_A}$ and is an immediate generalization of Proposition~\ref{prop:Gcontraction}.

\begin{lemma}\label{lem:hom_contraction}
Let $\zeta_A$ be a character which is homogeneous of degree $d$.
The symmetric function $\Psi_{\zeta_A}(\mathbf{g})$ is given by
\begin{equation*}\label{eq:Pda-prop}
\Psi_{\zeta_A}(\mathbf{g})=\sum_{\pi \in \Pi_A(\mathbf{g})} \Psi_{\bullet}(\mathbf{g}_{\downarrow\pi})(\mathbf{x}^d).
\end{equation*}
In particular, $\Psi_{\zeta_A}(\mathbf{g}) = 0$ if $d$ does not divide $n$.
\end{lemma}

\begin{proposition}
Let $\zeta_A$ be a character which is homogeneous of degree $d$.
Let $\mathbf{g}$ be a graph on $n = kd$ vertices.
Then $(-1)^{n - k}\omega(\Psi_{\zeta_A}(\mathbf{g}))$ is $p$-positive.
\label{prop:hom_omega}
\end{proposition}

\begin{proof}
By Lemma~\ref{lem:hom_contraction} we know that
\[\Psi_{\zeta_A}(\mathbf{g})=\sum_{\pi \in \Pi_A(\mathbf{g})} \Psi_{\bullet}(\mathbf{g}_{\downarrow\pi})(\mathbf{x}^d).\]
Furthermore, $\omega(\Psi_{\bullet}(\mathbf{g}_{\downarrow\pi}))$ is $p$-positive for each $\pi \in \Pi_A(\mathbf{g})$.
Also, $\mathbf{g}_{\downarrow\pi}$ is a graph on $k$ vertices.
We may write
\[\Psi_{\bullet}(\mathbf{g}_{\downarrow\pi}) = \sum_{\lambda \vdash k} (-1)^{k - \ell(\lambda)} a_{\lambda}p_{\lambda}(\mathbf{x})\]
where $a_{\lambda} \geq 0$.
For $\lambda \vdash k$ with $\lambda = (\lambda_1, \lambda_2, \dots, \lambda_{\ell})$ we let $d \cdot \lambda = (d\lambda_1, d\lambda_2, \dots, d\lambda_{\ell})$.
So, $d \cdot \lambda \vdash dk = n$ and $p_{\lambda}(\mathbf{x}^d) = p_{d \cdot \lambda}(\mathbf{x})$.
We then find
\begin{align*}
    (-1)^{n - k}\Psi_{\bullet}(\mathbf{g}_{\downarrow\pi})(\mathbf{x}^d) &= (-1)^{n - k}\sum_{\lambda \vdash k} (-1)^{k - \ell(\lambda)} a_{\lambda}p_{\lambda}(\mathbf{x}^d)\\
    &= (-1)^{n - k}\sum_{\lambda \vdash k} (-1)^{k - \ell(\lambda)} a_{\lambda}p_{d \cdot \lambda}\\
    &= \sum_{\lambda \vdash k} (-1)^{n - \ell(\lambda)} a_{\lambda}p_{d \cdot \lambda}
\end{align*}
and see that $\omega((-1)^{n - k}\Psi_{\bullet}(\mathbf{g}_{\downarrow\pi})(\mathbf{x}^d))$ is $p$-positive for each $\pi \in \Pi_A(\mathbf{g})$.
Therefore, $(-1)^{n - k}\omega(\Psi_{\zeta_A}(\mathbf{g}))$ is $p$-positive.
\end{proof}

Proposition~\ref{prop:hom_omega} combined with Lemma~\ref{lem:palt} gives us that $\pm \Psi_{\zeta_A}(\mathbf{g})$ is $h$-alternating.
So, Theorem~\ref{thm:nablapos} applies and after applying $\nabla$ and setting $q=1$ we obtain, up to a global sign, a Schur nonnegative symmetric function for any graph $\mathbf{g}$ and homogeneous character $\zeta_A$.
Hence, an analysis similar to what was done for the usual chromatic symmetric functions in Section~\ref{s:nabla} can be performed for any such $\Psi_{\zeta_A}(\mathbf{g})$.

We turn our attention to the polynomial $\phi_t \circ \Psi_{\zeta_A}(\mathbf{g})$ which we will denote by $\chi_{\mathbf{g}, \zeta_A}(t)$.
For any graph $\mathbf{g}$ let $\mathcal{F}(\mathbf{g})$ denote the collection of flats of $\mathbf{g}$ and let $a(\mathbf{g})$ denote the number of acyclic orientations of $\mathbf{g}$.
When $F$ is a subset of edges of a graph $\mathbf{g}$ with vertex set $V$ we let $\mathbf{g}|_{V,F}$ denote the graph on the same vertex set $V$ which has edge set $F$.
We also let $c(F)$ denote the number of connected components of $\mathbf{g}|_{V,F}$.
The antipode in $\overline{\mathcal{G}}$ has the formula
\begin{equation}
S(\mathbf{g}) = \sum_{F \in \mathcal{F}(\mathbf{g})} (-1)^{c(F)} a(\mathbf{g}/F) \mathbf{g}|_{V,F}
\label{eq:antipode}
\end{equation}
for any $\mathbf{g} \in \mathcal{G}$.
The formula is Equation~(\ref{eq:antipode}) is due to Humpert and Martin~\cite{Humpert-Martin}.
We use $\mathcal{F}(\mathbf{g},A)$ to denote the collection of all $F \in \mathcal{F}(\mathbf{g})$ such that every connected component of $\mathbf{g}|_{V,F}$ is in $A$.
Notice this means $F \in \mathcal{F}(\mathbf{g},A)$ if and only if $F \in \mathcal{F}(\mathbf{g})$ and the set partition of the vertex set of $\mathbf{g}$ given by the connected components of $\mathbf{g}|_{V,F}$ is an element of $\Pi_A(\mathbf{g})$.
In the case of a homogenous character, we now obtain a formula giving a combinatorial interpretation from $\chi_{\mathbf{g},\zeta_A}(-1)$ from the antipode.

\begin{proposition}
Let $\zeta_A$ be a character which is homogeneous of degree $d$.
Let $\mathbf{g}$ be a graph on $n = kd$ vertices, then
\[(-1)^k \chi_{\mathbf{g}, \zeta_A}(-1) = \sum_{F \in \mathcal{F}(\mathbf{g},A)} a(\mathbf{g}/F).\]
\label{prop:minus1}
\end{proposition}

\begin{proof}
We first compute to find
\begin{align*}
\chi_{\mathbf{g}, \zeta_A}(-1) &= \zeta_A\left( \sum_{F \in \mathcal{F}(\mathbf{g})} (-1)^{c(F)} a(\mathbf{g}/F) \mathbf{g}|_F \right)\\
&= \sum_{F \in \mathcal{F}(\mathbf{g})} (-1)^{c(F)} a(\mathbf{g}/F) \zeta_A(\mathbf{g}|_F)\\
&= \sum_{F \in \mathcal{F}(\mathbf{g},A)} (-1)^{c(F)} a(\mathbf{g}/F).
\end{align*}
The proposition then follows since $c(F) = k$ for all $F \in \mathcal{F}(\mathbf{g},A)$.
\end{proof}

\begin{remark}~\label{rem:polynomial}
We have given a Hopf algebraic proof in Proposition~\ref{prop:minus1}.
From Lemma~\ref{lem:hom_contraction} it follows for a homogeneous character $\zeta_A$ that $\chi_{\mathbf{g}, \zeta_A}$ is the sum of chromatic polynomials of contractions of $\mathbf{g}$.
Hence, by a theorem of Stanley~\cite{Sta1973} we obtain an interpretation of $\chi_{\mathbf{g}, \zeta_A}(-m)$ for any positive integer $m$.
\end{remark}

As observed in Remark~\ref{rem:polynomial} we have
\begin{equation}
\chi_{\mathbf{g}, \zeta_A}(t) = \sum_{\pi \in \Pi_A(\mathbf{g})} \chi_{\mathbf{g}_{\downarrow\pi}, \zeta_A}(t)
\label{eq:polynomial}
\end{equation}
by Lemma~\ref{lem:hom_contraction}.
This is because substituting $x_i$ with $x_i^d$ in a monomial quasisymmetric function gives another monomial quasisymmetric indexed by an integer composition of the same length.
The image $\phi_t(M_{\alpha})$ of the Hopf morphism $\phi_t: QSym \to \field[t]$ only depends on $\ell(\alpha)$.
It then follows immediately from Proposition~\ref{prop:Whitney} that the coefficients of $\chi_{\mathbf{g}, \zeta_A}(t)$ alternate in sign and have an interpretation in terms of non-broken circuits.
Hence, we arrive at the following natural question.

\begin{question}
Let $\zeta_A$ be a homogenous character, do the absolute values of the coefficients $\chi_{\mathbf{g}, \zeta_A}(t)$ for a graph $\mathbf{g}$ form a unimodal, or even log-concave, sequence?
\end{question}

We now discuss how the symmetric functions coming from a character $\zeta_A$ fit into larger class of problems.
Here we will not need any homogeneity assumption and allow $A$ to be any collection of connected graphs.
A \emph{scheduling problem} on $n$ elements is a Boolean formula in the atomic formulas $(x_i \leq x_j)$ for $i,j \in [n]$.
A function $f$ from $[n]$ to the positive integers is a solution to the scheduling problem if the Boolean formula evaluates to true when $x_i = f(i)$.
In~\cite{sched} Breuer and Klivans studied quasisymmetric functions associated to a scheduling problem.
Given such a function $f$ we obtain a monomial $\mathbf{x}_f = \prod_{i \in [n]} x_{f(i)}$.
Let $\mathbf{S}$ be any scheduling problem on $n$ elements, the quasisymmetric function associated to $\mathbf{S}$ by Breuer and Klivans is then
\[\Phi(\mathbf{S}) = \sum_f \mathbf{x}_f\]
where the sum in taken over all solutions to the scheduling problem.
There is also an associated quasisymmetric function in noncommuting variables.
It can be advantageous to work with noncommuting variables.
For example, there is no deletion-contraction law for the usual chromatic symmetric; however, there is a deletion-contraction law for the chromatic symmetric in noncommuting varaiables~\cite{GS} as well as for many other scheduling problems~\cite{MachacekEJC}.
It is thus desirable to know that a problem fits into the context scheduling problems.
One then obtains connections to geometric enumeration  through Ehrhart theory (see~\cite[Theorem 3.2]{sched}), and in addition, one obtain the algebraic quasisymmetric function invariant
associated to the scheduling problem in both commuting as well as  noncommuting variables.

Usual graph coloring is one example of a scheduling problem, and in the case of graph coloring the associated quasisymmetric is the chromatic symmetric function.
For $\mathbf{g} \in \mathcal{G}_n$ let $E$ denote its edge set and let $x_i$ represent the color assigned to the vertex $i$.
The scheduling problem
\[\bigwedge_{\{i,j\} \in E} (x_i \neq x_j)\]
then corresponds to properly coloring $\mathbf{g}$.
Now for any collection of connected graphs $A$ we define the scheduling problem
\[\mathbf{S}(\mathbf{g},A) = \bigvee_{F \in \mathcal{F}(\mathbf{g},A)}\left( \left(\bigwedge_{\{i,j\} \in F} (x_i = x_j) \right)\wedge  \left(\bigwedge_{\{i,j\} \in E \setminus F} (x_i \neq x_j) \right) \right)\]
where $E$ denotes the edge set of $\mathbf{g}$.
The next proposition follows directly from the definition of the quasisymmetric function associated to a scheduling problem and the construction of $\mathbf{S}(\mathbf{g},A)$.

\begin{proposition}
If $A$ is any collection of connected graphs and $\mathbf{g}$ is any graph, then
\[\Phi(\mathbf{S}(\mathbf{g},A)) = \Psi_{\zeta_A}(\mathbf{g}).\]
\end{proposition}


%
%
%
%
%
%
%
%
%
%
%
%


\small
\bibliographystyle{abbrv}  
\bibliography{PermInvariant}

\end{document}